\documentclass[a4paper,11pt]{article}
\usepackage[T1]{fontenc}
\usepackage[utf8]{inputenc}
\usepackage[text={160mm,248mm},centering]{geometry}
\usepackage{parskip,tikz-cd,microtype}
\usepackage{hyperref,amsthm,enumitem}

\usepackage[theoremfont]{stickstootext}
\usepackage[stix2,smallerops]{newtxmath}
\usepackage[cal=stixtwoplain]{mathalfa}
\tikzcdset{arrow style=tikz, diagrams={>={Straight Barb[scale=0.7]}}} 
\DeclareFontEncoding{FMX}{}{}
\DeclareFontSubstitution{FMX}{futm}{m}{n}
\DeclareSymbolFont{fourier}{FMX}{futm}{m}{n}
\let\sum\relax\DeclareMathSymbol{\sum}{\mathop}{fourier}{80}
\let\prod\relax\DeclareMathSymbol{\prod}{\mathop}{fourier}{81}

\numberwithin{equation}{section}
\theoremstyle{plain}
\newtheorem{theorem}[equation]{Theorem}
\newtheorem{lemma}[equation]{Lemma}
\newtheorem{corollary}[equation]{Corollary}
\theoremstyle{definition}
\newtheorem{situation}[equation]{Situation}
\newtheorem{remark}[equation]{Remark}
\newtheorem{question}[equation]{Question}
\newtheorem{definition}[equation]{Definition}
\newtheorem{example}[equation]{Example}

\makeatletter
\renewenvironment{proof}[1][\proofname] {\par\pushQED{\qed}\normalfont\topsep6\p@\@plus6\p@\relax\trivlist\item[\hskip\labelsep\bfseries#1\@addpunct{.}]\ignorespaces}{\popQED\endtrivlist\@endpefalse}
\makeatother

\renewcommand{\mathbb}{\mathbf}
\renewcommand{\setminus}{\mathbin{\rule[0.2em]{0.67em}{0.12em}}}%
\renewcommand{\geq}{\geqslant}
\renewcommand{\leq}{\leqslant}

\title{Hodge and Frobenius colevels of algebraic varieties}

\author{Daqing Wan%
  \and Dingxin Zhang}%

\date{April 5, 2025}

\begin{document}

\setcounter{tocdepth}{1}

\maketitle

\begin{abstract}
\noindent
We provide new, improved lower bounds for the Hodge and Frobenius
colevels of algebraic varieties (over $\mathbb{C}$ or over a finite
field) in all cohomological degrees.  These bounds are expressed in
terms of the dimension of the variety and multi-degrees of its
defining equations.  Our results lead to an enhanced positive answer
to a question raised by Esnault and the first author.

\medskip\noindent
2020 \textit{Mathematics Subject Classification.}
Primary: \href{https://zbmath.org/classification/?q=cc%3A14F20}{14F20};
Secondary:
\href{https://zbmath.org/classification/?q=cc%3A11M38}{11M38},
\href{https://zbmath.org/classification/?q=cc%3A14G15}{14G15},
\href{https://zbmath.org/classification/?q=cc%3A32S35}{32S35}.
\end{abstract}

\tableofcontents

\section{Background and motivation}

These notes present our investigation into the cohomological colevels
of algebraic varieties in both the Betti and \(\ell\)-adic contexts,
motivated by our previous \(p\)-adic findings~\cite{wz22}.  We provide
new, improved lower bounds for the Hodge and Frobenius colevels of
algebraic varieties (over \(\mathbb{C}\) or over the algebraic closure
of a finite field) across all cohomological degrees.

These bounds are formulated in terms of the dimension and the
multi-degrees of the defining equations of the variety.  For a
detailed explanation, refer to Theorems~\ref{theorem:beyond},
\ref{theorem:before}, and \ref{theorem:projective-case}.  These
theorems offer an enhanced affirmative response to a question
(Question~\ref{question}) posed by Esnault and the first author in
\cite{esnault-wan:divisibility}.

Before examining the specific results presented in Section
\ref{sec:improved-bounds}, this introductory section will review some
background material, and illustrate the various aspects of the
problems we are about to address.

\subsection{Hodge colevel and Frobenius colevel}

Let us recall the definitions of \emph{cohomological colevel} in two
different contexts.

\begin{definition}[Frobenius colevel]
Let \(k\) be an algebraic closure of the prime field
\(\mathbb{F}_{p}\), and let \(\ell \neq p\) be another prime number.
Let $X$ be a finite-type, separated scheme over $k$.  Assume
\(\mathbb{F}_{q}\) is a finite subfield of \(k\) such that \(X\) is
defined over \(\mathbb{F}_{q}\), i.e.,
\(X = X_{0} \otimes_{\mathbb{F}_{q}} k\) for some
\(\mathbb{F}_{q}\)-scheme \(X_{0}\).  Let \(F_{q}\colon X \to X\)
denote the geometric Frobenius (over \(\mathbb{F}_{q}\)).  The
operator \(F_{q}\) acts on the \(\ell\)-adic cohomology
\(\mathrm{H}_{c}^{i}(X;\mathbb{Q}_{\ell})\) by transporting the
structures.  We say that \(\mathrm{H}^i_{c}(X;\mathbb{Q}_{\ell})\) has
\emph{(Frobenius) colevel} \(\geq m\) if all Frobenius eigenvalues are
divisible by \(q^{m}\) in the ring of algebraic integers.

Note that the condition of having a Frobenius colevel \(\geq m\) is
geometric in nature: if \(X\) is defined over
\(\mathbb{F}_{q^{\prime}}\), another finite field of characteristic
\(p\), then the eigenvalues of \(F_{q}\) acting on
\(\mathrm{H}^{i}(X;\mathbb{Q}_{\ell})\) are divisible by \(q^{m}\) if
and only if the eigenvalues of \(F_{q^{\prime}}\) acting on
\(\mathrm{H}^{i}(X;\mathbb{Q}_{\ell})\) are divisible by
\(q^{\prime m}\).
\end{definition}

\begin{definition}[Hodge colevel]
Assume \(k = \mathbb{C}\).  Let $X$ be a finite-type, separated scheme
over $\mathbb{C}$, with associated analytic space $X^{\text{an}}$.
The compactly supported cohomology
\(\mathrm{H}^i_{c}(X^{\mathrm{an}};\mathbb{Q})\) is endowed with a
canonical mixed Hodge structure.  We say that
\(\mathrm{H}^{i}_{c}(X^{\mathrm{an}};\mathbb{Q})\) has \emph{(Hodge)
  colevel} \(\geq m\) if the Hodge filtration
\(F^{\bullet}\mathrm{H}^{i}_{c}(X^{\mathrm{an}};\mathbb{C})\) on
\(\mathrm{H}^{i}_{c}(X^{\mathrm{an}};\mathbb{C})\) satisfies the
condition
\(F^{m}\mathrm{H}^{i}_{c}(X^{\mathrm{an}};\mathbb{C}) =
\mathrm{H}^{i}_{c}(X^{\mathrm{an}};\mathbb{C})\).%
\footnote{ These colevel conditions can also be defined for
  ``abstract'' mixed Hodge structures or Frobenius modules.  Let
  \(r \in \mathbb{Z}\) and \(K\) be a field.  We say that a filtered
  \(K\)-vector space \((V, \mathrm{Fil}^{\bullet}V)\) has (Hodge)
  colevel \(\geq r\) if \(V = \mathrm{Fil}^{r}V\).  If \(V\) is a
  finite-dimensional \(K[F]\)-module, we say that \(V\) has
  (Frobenius) colevel \(\geq r\) if all the eigenvalues of \(F\) can
  be expressed as \(q^{r} \cdot \beta\), where \(\beta\) is some
  algebraic integer.

  Dually, we say that a filtered vector space
  \((V, \mathrm{Fil}^{\bullet}V)\) has (Hodge) level \(\leq r\) if
  \(\mathrm{Fil}^{r+1}V = 0\); we say that a finite-dimensional
  \(K[F]\)-module \(V\) has (Frobenius) level \(\leq r\) if, for any
  eigenvalue \(\alpha\) of \(F\), \(q^{r} \cdot \alpha^{-1}\) is an
  algebraic integer.  }
\end{definition}

In the sequel, we will use the notation \(\mathrm{H}^i_{c}(X)\) to
represent \(\mathrm{H}^i_{c}(X^{\mathrm{an}};\mathbb{Q})\) or
\(\mathrm{H}^i_{c}(X;\mathbb{Q}_{\ell})\) when the base field \(k\) is
either irrelevant to the discussion or clearly understood from the
context.

If \(X\) is proper over \(k\), then the compactly supported cohomology
\(\mathrm{H}^{\ast}_{c}(X)\) is the same as the usual cohomology.  In
such cases, we will omit the subscript \(c\).

\begin{example}%
The colevel of \(\mathrm{H}^{2i}(\mathbb{P}^n)\) is \(\geq i\).  If
\(X\) is a variety over \(k\) of dimension \(\leq n\), then the
colevel of \(\mathrm{H}_{c}^{2n}(X)\) is \(\geq n\).

When \(k = \mathbb{C}\) and \(X\) is both smooth and projective, the
statement that
\(\mathrm{H}^i_{c}(X^{\mathrm{an}};\mathbb{Q})=\mathrm{H}^i(X^{\mathrm{an}};\mathbb{Q})\)
has Hodge colevel \(\geq m\) is equivalent to the condition that
\(\mathrm{H}^{i-a}(X;\Omega^{a}_{X}) = 0\) for \(m>a\).  Here,
\(\Omega^{a}_{X} = 0\) if $a < 0$.
\end{example}

\subsection{Geometric colevel and cohomological colevel}

The cohomological colevel conditions discussed above are connected to
the well-known conjectures about the existence of algebraic cycles in
the following manner.  Let \(X\) be a smooth projective variety over
\(k\).  If there exists a subvariety \(Z\) of codimension \(\geq m\)
such that the map \(\mathrm{H}_Z^i(X) \to \mathrm{H}^i(X)\) is
surjective---where the left-hand side denotes the cohomology of \(X\)
supported on \(Z\)---then we say that the \emph{geometric colevel} of
\(\mathrm{H}^i(X)\) is \(\geq m\).  For the general definition of the
colevel filtration (\emph{la filtration par le coniveau}), the reader
is referred to \cite[\S10.1]{grothendieck:brauer-3} or
\cite[Exposé~XX, \S2]{sga7-2}.  It can be shown that if the geometric
colevel of \(\mathrm{H}^i(X)\) is \(\geq m\), then the cohomological
(i.e., Frobenius or Hodge) colevel of \(\mathrm{H}^i(X)\) is also
\(\geq m\). See \cite[p.~167]{grothendieck:brauer-3}.

\begin{example}%
Let \(X\) be a proper nonsingular variety over \(k\).  Let \(k(X)\)
denote the rational function field of \(X\), and let
\(\overline{k(X)}\) be an algebraic closure of \(k(X)\).  Suppose that
\[
\mathrm{CH}_0\left(X \times_k \operatorname{Spec}\overline{k(X)}\right) \otimes_{\mathbb{Z}} \mathbb{Q} \simeq \mathbb{Q}
\]
(for example, proper nonsingular rationally connected varieties,
including all Fano manifolds, satisfy this condition).  Then
\(\mathrm{H}^i(X)\) has geometric colevel \(\geq 1\) for any
\(i > 0\).  Consequently, depending on the context, the Frobenius or
Hodge colevel of \(\mathrm{H}^i(X)\) is also at least \(1\). For
further information, see \cite{esnault:0-cycle}.
\end{example}

Conversely, if the cohomological colevel of \(\mathrm{H}^i(X)\) is
\(\geq m\), the generalized Hodge conjecture in the Betti case, and
the generalized Tate conjecture in the \(\ell\)-adic case, state that
the geometric colevel is also at least \(m\), see
\cite[\S10]{grothendieck:brauer-3} and \cite{grothendieck:trivial}.

\begin{remark}
There is a more sophisticated version of the generalized Hodge
conjecture, also proposed by Grothendieck.  This version is applicable
to both singular and non-proper varieties and is formulated in terms
of motives.  It states that if \(M\) is an ``effective mixed motive''
whose de Rham realization \(M_{\mathrm{DR}}\) has colevel at least
\(\mu\) (with respect to the Hodge filtration), then \(M\) can be
expressed as \(N \otimes \mathbb{Q}(-\mu)\), where \(N\) is some
effective mixed motive and \(\mathbb{Q}(-\mu)\) denotes the Tate
twist.
\end{remark}

It is well known that the existence of a subvariety \(Z\) that
supports the cohomology (or the existence of \(N\) in the more
sophisticated conjecture) has significant arithmetic consequences.
For instance, when \(k = \mathbb{C}\) and \(X\) is smooth and
projective, the generalized Hodge conjecture suggests that a Hodge
colevel of \(\mathrm{H}^i(X)\) being \(\geq m\) implies that a
sufficiently general mod \(p\) reduction of \(X\) should also have
Frobenius colevel \(\geq m\). Let us elaborate on this for readers who
may not be familiar with this circle of ideas.

If \(\mathrm{H}^{i}(X)\) has Hodge colevel \(\geq m\), the generalized
Hodge conjecture predicts the existence of a subvariety \(Z\) of \(X\)
with codimension \(\geq m\) such that the following map is surjective:
\[
\mathrm{H}^{i}_{Z^{\mathrm{an}}}(X^{\mathrm{an}};\mathbb{Q}) \to \mathrm{H}^{i}(X^{\mathrm{an}};\mathbb{Q}).
\]
By the comparison theorem between étale and singular cohomology, the map of \(\ell\)-adic cohomology groups
\begin{equation}
\label{eq:geometric-colevel}
\mathrm{H}^{i}_Z(X;\mathbb{Q}_{\ell}) \to \mathrm{H}^{i}(X;\mathbb{Q}_{\ell})
\end{equation}
is also surjective.  Now, let us consider a spread-out of \(X\), i.e.,
a smooth projective morphism \(\mathcal{X} \to \operatorname{Spec}R\)---where
\(R\) is a finitely generated subring of \(\mathbb{C}\) which is
smooth over \(\mathbb{Z}\)---such that
\(\mathcal{X} \otimes_R \mathbb{C} = X\).  Note that the residue
fields of the maximal ideals of \(R\) are finite fields.

Upon replacing \(\operatorname{Spec}R\) by a smaller nonempty open
affine, we may assume the Zariski closure \(\mathcal{Z}\) of \(Z\) is
flat over \(R\).  For any closed point \(b\) of
\(\operatorname{Spec}R\), we form the following cartesian diagram
\begin{equation*}
\begin{tikzcd}
Z \ar[hook,d] \ar[r] & \mathcal{Z} \ar[d,hook] & Z_{b} \ar[l] \ar[hook,d] \\
X \ar[d]\ar[r] & \mathcal{X}\ar[d] & X_b \ar[l] \ar[d]\\
\operatorname{Spec}\mathbb{C} \ar[r] & \operatorname{Spec}R & \{b\}\ar[l]
\end{tikzcd}.
\end{equation*}
We shall refer to \(X_b\) as a \emph{specialization} of \(X\).  Let
\(\overline{b}\) denote a geometric point of
\(\operatorname{Spec}(R)\) valued in an algebraic closure of the
residue field \(\kappa(b)\) of the closed point \(b\).

By generic base change~\cite[Finitude]{deligne:sga4.5}, there exists a
Zariski open dense subscheme of \(\operatorname{Spec} R\) over which
the formation of
\(\mathrm{H}^{i}_{Z_{\overline{b}}}(X_{\overline{b}},
\mathbb{Q}_{\ell})\) commutes with arbitrary base change.  Therefore,
for \(b\) lying in this open set, we have isomorphisms of vector
spaces:
\[
\mathrm{H}^{i}_{Z}(X; \mathbb{Q}_{\ell}) \simeq \mathrm{H}^{i}_{Z_{\overline{b}}}(X_{\overline{b}}; \mathbb{Q}_{\ell}),
\]
and
\[
\mathrm{H}^{i}(X; \mathbb{Q}_{\ell}) \simeq \mathrm{H}^{i}(X_{\overline{b}}; \mathbb{Q}_{\ell}).
\]
Therefore, the surjectivity
of \eqref{eq:geometric-colevel} implies that, for a ``sufficiently general'' closed point \(b\)
of \(\operatorname{Spec}R\), the map
\begin{equation*}
\mathrm{H}^{i}_{Z_{\overline{b}}}(X_{\overline{b}};\mathbb{Q}_{\ell}) \to \mathrm{H}^{i}(X_{\overline{b}};\mathbb{Q}_{\ell})
\end{equation*}
is surjective as well. Thus,
\(\mathrm{H}^{i}(X_{\overline{b}};\mathbb{Q}_{\ell})\) has geometric colevel
\(\geq m\); and it follows that \(\mathrm{H}^{i}(X_{\overline{b}};\mathbb{Q}_{\ell})\) has Frobenius colevel \(\geq m\), i.e.,
the Frobenius eigenvalues of \(X_{b}\) are divisible
by \(q^{m}\), where \(q\) is the cardinality of \(\kappa(b)\).

In particular, if \(\mathrm{H}^i(X;\mathcal{O}_X) = 0\) for \emph{all}
\(i > 0\), and if we assume the generalized Hodge conjecture is valid,
then a sufficiently general specialization \(X_{b}\) should have a
rational point according to Grothendieck's trace formula.
Furthermore, the number of \(\kappa(b)\)-points is congruent to \(1\)
modulo \(q\), where \(q = |\kappa(b)|\).

This prediction for smooth projective \(X\) over \(\mathbb{C}\), based
on the generalized Hodge conjecture, aligns with a more refined
conjecture---``Newton above Hodge''---formulated by
Katz~\cite[Conjecture 2.9]{katz:on-theorem-of-ax}.  This latter
conjecture has been unconditionally proven by Mazur
\cite{mazur:frobenius-and-hodge-filtration}.

\begin{remark}
One cannot expect the Hodge colevel of \(\mathrm{H}^{i}(X)\) to be
\emph{equal} to the Frobenius colevel of \(\mathrm{H}^{i}(X_b)\) for a
sufficiently general smooth specialization \(X_{b}\).  A well-known
example illustrates this: According to a theorem by
Elkies~\cite{elkies}, any elliptic curve over \(\mathbb{Q}\) has
infinitely many supersingular reductions.  Therefore, for any elliptic
curve \(E\) over \(\mathbb{Z}[1/N]\), there are infinitely many primes
\(p\) such that
\(\mathrm{H}^2(E_{\overline{\mathbb{F}}_p} \times E_{\overline{\mathbb{F}}_p})\) has
Frobenius colevel \(1\), while the Hodge colevel of
\(\mathrm{H}^2(E_{\mathbb{C}} \times E_{\mathbb{C}})\) is zero.

On the other hand, one does expect that the Hodge colevel of
\(\mathrm{H}^{i}(X)\) should be equal to the Frobenius colevel of
\(\mathrm{H}^{i}(X_b)\) for \emph{infinitely} many specializations
\(X_{b}\) (more generally, it is expected that \(X\) should have
infinitely many ``ordinary'' reductions).  The case of colevel zero
corresponds to the weak ordinarity conjecture proposed by Mustaţă and
Srinivas~\cite{mustata-srinivas}.  Proving this in general appears to
be quite challenging.  According to the generalized Tate conjecture,
this expectation is related to the potential equality between the
geometric colevel of \(\mathrm{H}^{i}(X)\) and the geometric colevel
of \(\mathrm{H}^{i}(X_b)\) for infinitely many specializations
\(X_{b}\).
\end{remark}

When \(X_{b}\) is a \emph{singular} specialization, the colevel bounds on \(X\)
still have important arithmetic implications. We refer the reader to
Esnault~\cite{esnault:deligne-integrality-rational-points}
and Berthelot--Esnault--Rülling~\cite{berthelot-esnault-rulling:coniveau-one-mixed-char}
for investigations around this theme.

\subsection{Common lower bounds for cohomological colevels}

More generally, for \emph{any} (without smooth proper condition) flat
morphism \(\mathcal{X} \to \operatorname{Spec}R\), with
\(R \subset \mathbb{C}\) finitely generated over \(\mathbb{Z}\), and
for \(b\) being ``sufficiently general'', the Hodge colevels of \(X\)
are not larger than the Frobenius colevels of \(X_{b}\)\footnote{ This
  can be seen, for instance, by spreading out a proper hypercovering
  of the generic fiber of \(\mathcal{X}\to \operatorname{Spec}R\),
  whose entries are normal crossing complements of smooth projective
  varieties.

  However, one cannot expect this to hold for every \(b\), even when
  the generic fiber of \(\mathcal{X} \to \operatorname{Spec}R\) is
  smooth and projective. For a counterexample, see
  \cite[Theorem~1.2]{esnault-xu}.  }.

This conceptual explanation, appealing as it is, is not very
practical.  For instance, it cannot even be used to deduce the
classical theorem of Ax and
Katz~\cite{ax:q-divisible,katz:on-theorem-of-ax} (on Frobenius
colevels of primitive cohomology of a \emph{possibly singular}
projective variety) from a Hodge-theoretic result
(cf.~\cite{esnault-nori-srinivas}): we could indeed lift the defining
equations, but we have no idea whether the prime we start with is
``sufficiently general''.

Therefore, it is desirable to establish \emph{common} lower bounds for
the Hodge colevel of \(\mathrm{H}^i_{c}(X)\) and the
Frobenius colevel of \(\mathrm{H}^{i}_{c}(X_{b})\),
regardless of whether \(b\) is general or not, and irrespective of how
singular the specialization \(X_b\) may be.  Ideally, this should be
done in a purely formal manner, so that the argument is applicable in
both the Frobenius and Hodge contexts.  Such lower bounds have been
extensively studied in the literature, with a notable example being
the following integrality theorem of Deligne.

\begin{theorem}[Deligne]%
  \label{theorem:deligne}
  Let \(k\) be either a finite field, or \(\mathbb{C}\).
  Let \(X\) be a finite type, separated scheme over \(k\) of dimension
  \(n\). Then for any integer \(j\in \mathbb{Z}\), \(\mathrm{H}^{n+j}_{c}(X)\)
  has colevel \(\geq \max\{0,j\}\).%
  \footnote{The finite field case was proved by Deligne in
    \cite[Exposé~XXI, Corollarie 5.5.3]{sga7-2}.
    The theorem for Hodge colevels (which is not needed in this paper) can be shown as follows.
    When \(X\) is nonsingular, the theorem follows from Poincaré duality and
    Deligne's \emph{definition} of Hodge filtration in terms of logarithmic forms;
    if \(X\) is possibly singular, one proceeds by induction on dimension,
    and use the long exact sequence
    \[
      \cdots \to \mathrm{H}^{n+j}_{c}(X\setminus X_{\mathrm{sing}}) \to \mathrm{H}^{n+j}_{c}(X) \to
      \mathrm{H}^{n+j}_{c}(X_{\mathrm{sing}}) \to \cdots,
    \]
    where \(X_{\mathrm{sing}}\) is the singular locus of the reduced structure of
    \(X\), which has lower dimension. Applying the theorem for nonsingular variety
    to \(X\setminus X_{\mathrm{sing}}\), and inductive hypothesis to
    \(X_{\mathrm{sing}}\) finishes the proof.
  }
\end{theorem}

The bounds provided by Deligne's theorem are quite crude, as they
apply to \emph{any} variety over \(k\).  To improve upon this, one
could consider the type of polynomial equations that define \(X\) and
investigate whether the colevels can be bounded from below by
quantities that depend solely on these equations.  One set of such
quantities consists of the following numbers.  These numbers appear,
for instance, in the classical Ax–Katz theorem
\cite{ax:q-divisible,katz:on-theorem-of-ax} and in the first author's
work~\cite{wan:poles} on the ``polar part'' of zeta functions of
algebraic varieties.

\begin{definition}\label{definition:mu}
For a sequence of integers \(d_1\geq \cdots \geq d_r \geq 1\) and
integers $N,j\geq 0$, define
\begin{equation*}
\mu_j(N;d_1,\ldots,d_r) = j + \max\left\{ 0, \left\lceil \frac{N -j - \sum_{i=1}^{r} d_i}{d_1}\right \rceil \right\}.
\end{equation*}
Note that we have
\(\mu_j(N;d_1,\ldots,d_r) = j + \mu_0(N-j;d_1,\ldots,d_r) \geq j\),
and the sequence \(\mu_j(N;d_1,\ldots,d_r)\) is increasing in $j$:
\[\mu_0(N;d_1,\ldots,d_r) \leq \mu_1(N;d_1,\ldots,d_r) \leq
\mu_2(N;d_1,\ldots,d_r) \leq \cdots.\]
\end{definition}

The following theorem summarizes the best known lower bounds in this
regard, which are the result of the contribution from several
mathematicians, including Ax, Katz, Deligne, Dimca, Esnault, Nori,
Srinivas, del~Angel, the first author, among others.

(For a closed subvariety \(Y\) of \(\mathbb{P}^N\), set
\(\mathrm{H}^{\ast}(Y)_{\mathrm{prim}}=\operatorname{Coker}(\mathrm{H}^{\ast}(\mathbb{P}^N) \to \mathrm{H}^{\ast}(Y))\).)

\begin{theorem}%
\label{theorem:previous}
Let \(k\) be either an algebraic closure of a finite field, or the field \(\mathbb{C}\) of complex numbers.
\begin{enumerate}
\item\emph{(Affine Case).} Let \(X\) be the affine variety in \(\mathbb{A}^N\) of dimension \(n\)
defined by \(f_1 = \cdots = f_r = 0\), where \(f_i \in k[x_1,\ldots,x_n]\),
has degree \(\leq d_i\).
\begin{enumerate}
\item  The colevels of \(\mathrm{H}^{\ast}_{c}(X)\) and
\(\mathrm{H}^{\ast}_{c}(\mathbb{A}^{N}\setminus X)\) are
at least \(\mu_0(N;d_1,\ldots,d_r)\), and
\label{item-1a}
\item for any \(0 \leq j\leq N-1\),
the colevels of
\[
\mathrm{H}^{N-1+j}_{c}(X) \quad \text{and}\quad \mathrm{H}^{N+j}_{c}(\mathbb{A}^N \setminus X)
\]
are at least \(\mu_j(N;d_1,\ldots,d_r)\).
\label{item-1b}
\end{enumerate}
\item\emph{(Projective Case).}
Let \(Y\) be the projective variety in \(\mathbb{P}^N\) of dimension \(n\)
defined by \(g_1 = \cdots = g_r = 0\), where
\(g_i \in k[x_0,\ldots,x_N]\) is homogeneous of degree \(\leq d_i\).
\begin{enumerate}
\item The colevels of \(\mathrm{H}^{\ast}(Y)_{\mathrm{prim}}\) and
\(\mathrm{H}_c^{\ast}(\mathbb{P}^{N}\setminus Y)\) are
at least \(\mu_{0}(N+1;d_1,\ldots,d_r)\), and
\label{item-2a}
\item  for \(0\leq j \leq n\), the colevels of \(\mathrm{H}^{n+j}(Y)_{\mathrm{prim}}\) and
\(\mathrm{H}^{n+1+j}_{c}(\mathbb{P}^N\setminus{Y})\)
are at least \(\mu_j(N+1;d_1,\ldots,d_r)\).
\label{item-2b}
\end{enumerate}
\end{enumerate}
\end{theorem}

\subsection{A brief history of Theorem~\ref{theorem:previous}}
Item (\ref{item-1a}) for Frobenius colevel was inspired by the
classical theorem of Ax and
Katz~\cite{ax:q-divisible,katz:on-theorem-of-ax}, which asserts that
the reciprocal zeros and poles of the zeta function \(Z_{X}(t)\)
of \(X\) are divisible by \(\mu_0(N; d_1, \ldots, d_{r})\) in the ring
of algebraic integers.  But the Ax–Katz theorem does not directly
imply (\ref{item-1a}), as there may be cancellation of Frobenius
eigenvalues in the zeta function.  The form as stated is due to
Esnault and Katz~\cite[Theorems~2.1--2.3 and
\S5.3]{esnault-katz:cohomological-divisibility}, based on a method
developed by
del~Angel~\cite{del-angel:hodge-type-projective-varieties-low-degree}. This
lower bound is essentially optimal in general, as the Ax–Katz bound is
sharp.

Item (\ref{item-2a}) for Hodge colevels was proved by Deligne
in~\cite[Exposé~XI]{sga7-2} when $Y$ is a smooth complete
intersection, and historically this result motivated Katz's
improvement of Ax's theorem.  For an arbitrary hypersurface $Y$,
(\ref{item-2a}) for Hodge colevel was proved by Deligne and
Dimca~\cite{deligne-dimca:hodge-filtration-and-pole-order-filtration}.
For an arbitrary complete intersection $Y$, (\ref{item-2a}) for Hodge
colevel was proved by Esnault~\cite{esnault:hodge-type-small-degrees}.
For an arbitrary projective $Y$, (\ref{item-2a}) for Hodge colevel was
proved by Esnault--Nori--Srinivas~\cite{esnault-nori-srinivas}.

For Frobenius colevels,
Esnault~\cite{esnault:frobenius-complete-intersections-low-degree}
proved (\ref{item-2a}) for complete intersections.  The general result
(\ref{item-2a}) is proved by Esnault and
Katz~\cite[Theorem~4.1(1)]{esnault-katz:cohomological-divisibility}.
This part of the estimate is essentially sharp in general, again by
Ax--Katz and by Deligne's result for smooth complete intersections
above.

As hinted by Deligne's Theorem \ref{theorem:deligne}, compactly
supported cohomology groups beyond the middle dimension tend to have
larger colevels.  Thus, Items (\ref{item-1b}) and (\ref{item-2b}) both
reflect this phenomenon.

In a different related direction, the first author~\cite{wan:poles}
made some early attempts in the framework of zeta functions via
Dwork's $p$-adic theory.  He showed that, in the affine case, the
reciprocal poles of $Z_{X}(t)^{(-1)^{\dim X+1}}$ are divisible by
$q^{\mu_{1}(N;d_1,\ldots,d_r)}$ as algebraic integers, when $X$ is a
complete intersection.  This improved the ``polar part'' of the
Ax--Katz theorem for complete intersections.

Motivated by this, (\ref{item-2b}) for Hodge colevel was shown by
Esnault and the first author
\cite{esnault-wan:hodge-type-exotic-cohomology} for projective
complete intersection $Y$.
Esnault--Katz~\cite{esnault-katz:cohomological-divisibility} proved,
for an arbitrary projective $Y$, that the colevel of
$\mathrm{H}^{N+j}_{c}(\mathbb{P}^{N}\setminus Y)$ is
$\geq \mu_{j}(N+1;d_1,\ldots,d_r)$.  This is more general but weaker
than (\ref{item-2b}), as the cohomological degree of the improved
bound starts from the bigger middle dimension $N$ of
$\mathbb{P}^{N}\setminus Y$, not the smaller number $n+1$. The sharper
form (\ref{item-2b}) in the general case was proved by Esnault and the
first author~\cite{esnault-wan:divisibility}.

\subsection{Better colevel bounds for affine varieties} Similarly,
Item (\ref{item-1b}) is not an optimal estimate either, unless $X$ is
a hypersurface, as the cohomological degree of the improved bound
starts from the bigger number $N-1$, not the smaller middle dimension
$n$ of $X$.

In \cite{esnault-wan:divisibility}, the following improvement of
(\ref{item-1b}) was proposed as an open question. It is the affine
analogue of the better bound (\ref{item-2b}) for projective varieties:

\begin{question}[\cite{esnault-wan:divisibility}]\label{question}
  Under the hypothesis
  of Theorem~\ref{theorem:previous}(1), is it true that
  for \(0 \leq j \leq n\),
  the colevels of \(\mathrm{H}^{n +j}_{c}(X)\) and
  \(\mathrm{H}_{c}^{n+j+1}(\mathbb{A}^N\setminus X)\) are at least
  \(\mu_j(N;d_1,\ldots,d_r)\)?
\end{question}

In \cite{esnault-wan:divisibility}, Esnault and the first author gave
an evidence that this question should have a positive answer: if the
defining equations of \(X\) have a certain good behavior at infinity,
then one can reduce the bounds proposed in the question to
Theorem~\ref{theorem:previous}(2b).  It is also noted that if a
suitable version of weak Lefschetz theorem holds, then the Question
has an affirmative answer.

In these notes, we prove the needed Lefschetz type theorem.  See
Lemma~\ref{lemma:gysin}.  As a consequence, we get the
following\footnote{After we finish the manuscript, we learned that Rai
  and Shuddhodan also obtained a proof of
  Theorem~\ref{theorem:positive}~\cite{rs23}.}.

\begin{theorem}[See Lemma~\ref{lemma:complete-intersection-bound}]%
  \label{theorem:positive}
  Question~\ref{question} has an affirmative answer.
  That is, for integer $0 \leq j\leq n$,
  the colevels of \(\mathrm{H}^{n +j}_{c}(X)\)
  and \(\mathrm{H}_{c}^{n+j+1}(\mathbb{A}^N\setminus X)\) are at least
  \(\mu_j(N;d_1,\ldots,d_r)\).
\end{theorem}

Since $\mu_j(N;d_1,\ldots,d_r)\geq j$, Theorem~\ref{theorem:positive}
improves Deligne's integrality theorem (Theorem~\ref{theorem:deligne})
for affine variety $X$.  Also, this theorem gives a cohomological
interpretation of the polar result in~\cite{wan:poles} on zeta
functions of complete intersections.

While Theorem~\ref{theorem:positive} confirms Question~\ref{question},
the bounds are still not optimal.  Moreover, it does not provide
information of colevels of \(\mathrm{H}^{\ast}_{c}(X)\)
before the middle cohomological dimension if \(X\) is not a complete
intersection.  We shall establish new lower bounds of the colevels of
\(\mathrm{H}^{i}_{c}(X)\) for \emph{all} \(i\).  See
Theorems~\ref{theorem:beyond} and \ref{theorem:before}.  For the
cohomology groups beyond middle dimension, our new bounds will be
sharper than the ones proposed in the question above and hence better
than both Theorem~\ref{theorem:previous}(\ref{item-1b}) and Theorem
\ref{theorem:positive}; for cohomology groups before middle dimension,
our lower bounds will improve the Ax--Katz type bound of
Theorem~\ref{theorem:previous}(\ref{item-1a}).  Examples
\ref{example:det-1} and \ref{example:det-2} in the final section give
concrete illustrations of how our new bounds are improving the bounds
provided by Theorem~\ref{theorem:previous} and
Theorem~\ref{theorem:positive}.

From our refined colevel bounds for affine
varieties, one deduces formally ``projective bounds'' improving
Theorem~\ref{theorem:previous}(\ref{item-2a}, \ref{item-2b}). See Theorem~\ref{theorem:projective-case}.

These refined lower bounds were first discovered in our previous
Dwork-theoretic study of Frobenius eigenvalues of rigid cohomology
\cite{wz22}. We noticed there that our \(p\)-adic method is not
``formal'' enough to provide colevel bounds for \(\ell\)-adic and
Betti cohomology.

The proof we present in these notes combines several ideas: the
``Gysin lemma'' (Lemma~\ref{lemma:gysin}) already mentioned; a
dévissage designed by
del~Angel~\cite{del-angel:hodge-type-projective-varieties-low-degree}
and Esnault--Katz~\cite{esnault-katz:cohomological-divisibility} which
allows us to reduce the number of defining equations without
jeopardizing the lower bound, see
Lemma~\ref{lemma:almost-complete-intersection-colevel}; and finally a
crucial algebraic lemma devised in \cite{wz22} which shows that the
``core'' of the improved lower bounds comes from a set-theoretic
complete intersection, see Lemma~\ref{lemma:key}.

\section{Improved colevel bounds}
\label{sec:improved-bounds}

To state our refined colevel bounds, let us introduce
some sequences of numbers.

\begin{definition}\label{definition:nu}
Let \(d_1, d_2, \ldots, d_r\) and \(N\) be positive integers.  Assume
that \(d_1 \geq d_2 \geq \cdots \geq d_r\).  For each integer
\(e \in \mathbb{Z}\), and \(i = 1,2,\ldots,r\), we define
\begin{equation*}
d_i^{\ast}(e) =
\begin{cases}
  d_i, & \text{if } i\leq e, \\
  1,   & \text{if } i > e, \text{ and } d_i = d_1, \\
  0, &  \text{if } i > e, \text{ and }d_i < d_1.
\end{cases}
\end{equation*}
So if \(e \leq 0\), \(d_{i}^{\ast}(e) = 1\) or \(0\) depending whether
\(d_{i} = d_{1}\) or \(d_{i} < d_{1}\).  If \(e \geq r\), then
\(d_{i}^{\ast}(e) = d_{i}\).

Define
\[
\nu_j^{(e)}(N;d_1,\ldots,d_r) = j+
\max\left\{0, \left\lceil  \frac{N-j - \sum_{i=1}^{r}d_i^\ast(e)}{d_1}\right\rceil\right\}.
\]
\end{definition}
Since $d_i^*(e)$ is increasing in $e$, it is clear that $\nu_j^{(e)}(N;d_1,\ldots,d_r)$ is decreasing in $e$.
This gives the decreasing relation
\begin{align*}
  \cdots &= \nu^{(-1)} (N;d_{1},\ldots,d_{r}) = \nu^{(0)}(N;d_{1},\ldots,d_{r}) \\
         &\geq \nu_j^{(1)}(N;d_1,\ldots,d_r) \geq \cdots
           \geq \nu_j^{(r)}(N;d_1,\ldots,d_r)
           = \nu_{j}^{(r+1)}(N;d_{1},\ldots,d_{r}) = \cdots \\
         &= \mu_j(N;d_1,\ldots,d_r).
\end{align*}

Our improved colevel bounds beyond middle dimension is the following.

\begin{theorem}%
  \label{theorem:beyond}
  Let \(X\) be an \(n\)-dimensional Zariski closed subset of \(\mathbb{A}^N\)
  defined by polynomials \(f_1 = \cdots = f_r = 0\).
  Let \(d_1\geq \cdots \geq d_r\)
  be a sequence of positive integers such that \(d_i \geq \deg f_i\).
  Then the colevels of \(\mathrm{H}^{n+j}_{c}(X)\) and \(\mathrm{H}_{c}^{n+j+1}(\mathbb{A}^N\setminus X)\)
  are  at least
  \(\nu_j^{(N-n)}(N;d_1,\ldots,d_r)\) for all $0\leq j\leq n$.
\end{theorem}

Our improved colevel bounds before middle dimension is the following.

\begin{theorem}%
  \label{theorem:before}
  In the situation of Theorem \ref{theorem:beyond},
  for \(0\leq m \leq r-(N-n)\), the colevels of
  \(\mathrm{H}^{N-r+m}_{c}(X)\) and \(\mathrm{H}_{c}^{N-r+m+1}(\mathbb{A}^N\setminus X)\) are at least
  \(\nu_0^{(r-m)}(N;d_1,\ldots,d_r)\).
\end{theorem}

Under the hypothesis of Theorems~\ref{theorem:beyond} and  \ref{theorem:before},
we have
\(\mathrm{H}^{i}_{c}(X) = 0\) for \(i < N-r\)
(this is well-known, an elementary argument is given in \cite[Lemma~6.2]{wz22}); also by general facts about cohomology,
we have
\(\mathrm{H}^{n+j}_{c}(X) = 0\) for \(j > n\).
Thus the two theorems cover all possibly nonzero compactly supported cohomology groups of
\(X\). The integer $r-(N-n)$ is always non-negative. It is zero if and only if
$X$ is a set-theoretic complete
intersection by \(r\) equations.

The assertions for the complement \(\mathbb{A}^N \setminus X\) follow from
the results about \(X\) by using the long exact sequence
\begin{equation*}
  \cdots \to \mathrm{H}^{i}_{c}(\mathbb{A}^N \setminus X) \to
  \mathrm{H}^{i}_{c}(\mathbb{A}^N) \to \mathrm{H}_{c}^{i}(X) \to \cdots
\end{equation*}
and the well-known fact that
\begin{equation*}
  \mathrm{H}^{i}_{c}(\mathbb{A}^N) =
  \begin{cases}
    0, & i \neq 2N, \\
    \Lambda(-N), & i = 2N.
  \end{cases}
\end{equation*}
Thus, Theorems~\ref{theorem:beyond} and \ref{theorem:before}
cover all possibly nonzero compactly supported cohomology of \(\mathbb{A}^N\setminus X\)
except the top-dimensional (\(i=2N\)) one, which is trivially of colevel \(N\).

Since $\nu_j^{(N-n)}(N;d_1,\ldots,d_r) \geq \mu_j(N;d_1,\ldots,d_r)$,
we see Theorem \ref{theorem:beyond} implies, under its hypothesis, that

\begin{corollary}
The colevel of \(\mathrm{H}^{n+j}_{c}(X)\) is at least
\(\mu_j(N;d_1,\ldots,d_r)\).
\end{corollary}

Thus, Theorem~\ref{theorem:beyond} gives an affirmative answer to Question~\ref{question} and hence also
improves Theorem~\ref{theorem:positive}. For if $X$ is a complete intersection by given \(r\) equations, namely, $r=N-n$, one has
$\nu_j^{(N-n)}(N;d_1,\ldots,d_r) = \mu_j(N;d_1,\ldots,d_r)$. In this case, Theorem~\ref{theorem:beyond}
is the same as Theorem~\ref{theorem:positive}. Thus, strict improvements of Theorem~\ref{theorem:beyond}
over Theorem~\ref{theorem:positive} can only occur for non-complete intersections.
See Examples~\ref{example:det-1} and \ref{example:det-2}.

Similarly, since $\nu_0^{(r-m)}(N;d_1,\ldots,d_r) \geq \mu_0(N;d_1,\ldots,d_r)$, the lower bounds given by Theorem \ref{theorem:before}
are improvements of the Ax--Katz type bound \(\mu_0(N;d_1,\ldots,d_r)\).

From Theorems~\ref{theorem:beyond} and \ref{theorem:before}, we can formally
deduce the following projective version which is an improvement of Theorem~\ref{theorem:previous}(\ref{item-2a}, \ref{item-2b}).

\begin{theorem}%
\label{theorem:projective-case}
Let \(Y\) be a Zariski closed subset of \(\mathbb{P}^{N}\) of dimension \(n\),
defined by homogeneous polynomial equations \(g_1 = \cdots = g_r = 0\).
Assume that \(\deg g_i \leq d_i\), \(d_1 \geq \cdots \geq d_r\).
Then
\begin{itemize}
\item for all integers \(0 \leq j \leq n\), the colevels of \(\mathrm{H}^{n+j}(Y)_{\mathrm{prim}}\) and
\(\mathrm{H}^{n+j+1}_{c}(\mathbb{P}^N\setminus Y)\) are at least \(\nu^{(N-n)}_{j}(N+1;d_1,\ldots,d_r)\);
\item for all integers \(0\leq m \leq r-(N-n)\), the colevels of \(\mathrm{H}^{N-r+m}(Y)_{\mathrm{prim}}\) and
\(\mathrm{H}^{N-r+m+1}_{c}(\mathbb{P}^N\setminus Y)\) are at least \(\nu^{(r-m)}_{0}(N+1;d_1,\ldots,d_r)\).
\end{itemize}
\end{theorem}

\medskip It is worth remarking that we \emph{do not} give new proofs of
Theorems~\ref{theorem:previous}(\ref{item-1a}, \ref{item-2a}) (which are
ultimately based on the Ax--Katz theorem, a theorem whose known proofs are all
\(p\)-adic in nature).  Rather, these results are the starting point for our proofs of
Theorems~\ref{theorem:beyond} and \ref{theorem:before}.

As an arithmetic application, we give the following corollary which improves the ``polar'' part of
the Ax--Katz theorem for the zeta function of an arbitrary affine variety, extending the complete intersection case treated in~\cite{wan:poles}.
This result can also be derived from the corresponding Frobenius colevel bounds in rigid cohomology \cite{wz22}.

\begin{corollary}
\label{corollary:polar}
In the situation of Theorem \ref{theorem:beyond},  assume now that $k$ is a finite field of $q$ elements.
\begin{enumerate}
\item If $r=N-n$ (complete intersection), then all reciprocal poles of the zeta function $Z_X(t)^{(-1)^{N-r+1}}$
are divisible, as algebraic integers, by $q^{\mu_1(N; d_1,\cdots, d_r)}$.

\item If $r > N-n$ (non--complete intersection),  then all reciprocal poles of the zeta function $Z_X(t)^{(-1)^{N-r+1}}$
are divisible, as algebraic integers, by $q^{\nu_0^{(r-1)}(N; d_1,\cdots, d_r)}$.
\end{enumerate}
\end{corollary}

\begin{proof} The following sequence of colevel bounds is increasing:
\[
\nu_0^{(r)}(N; d_1, \cdots,d_r) \leq  \cdots \leq \nu_0^{(N-n)}(N; d_1, \cdots, d_r) \leq \nu_1^{(N-n)}(N; d_1, \cdots, d_r)
\leq \cdots .
\]
By the trace formula, the zeta function $Z_X(t)^{(-1)^{N-r+1}}$ is given by the product
\[
\prod_{m=0}^{r-(N-n)} \det(I-t\cdot F\mid\mathrm{H}_{c}^{N-r+m}(X))^{(-1)^m}
\prod_{j=1}^{n} \det(I-t\cdot F\mid\mathrm{H}_{c}^{n+j}(X))^{(-1)^{N-r+n+j}},
\]
where $F$ is the geometric Frobenius map. The factor for $m=0$ is clearly in
the numerator. By Theorems~\ref{theorem:beyond} and \ref{theorem:before},
the reciprocal roots of all other factors
(including all polar factors) are divisible, as algebraic integers,
by $q^{\nu_0^{(r-1)}(N; d_1,\cdots, d_r)}$ ($m=1$) if $r>N-n$, and
by $q^{\nu_1^{(N-n)}(N; d_1,\cdots, d_r)}$,
which is $q^{\mu_1(N; d_1,\cdots, d_r)}$, if $r=N-n$.
\end{proof}

We close this section by proving the following very elementary lemma,
which summarizes some basic properties of
the numbers \(\nu_j^{(e)}(N;d_1,\ldots,d_r)\).
These properties will be used in our proof.

\begin{lemma}%
\label{lemma:properties-of-nu}
The numbers \(\nu_j^{(e)}(N;d_1,\ldots,d_r)\) satisfy the following
properties.
\begin{enumerate}
\item \label{property-1}  \(\nu^{(r)}_j(N;d_1,\ldots,d_r) = \mu_j(N;d_1,\ldots,d_r)\).
\item \label{property-2} \(\nu_{j}^{(e)}(N;d_1,\ldots,d_r)\) is decreasing in \(e\) and \(d_i\),
increasing in \(j\).
\item \label{property-3} For any \(e \in \mathbb{Z}\), \(i \geq 1\),
\(j \geq 0\), we have
\[\nu^{(e)}_{j+i}(N;d_1,\ldots,d_r) \geq \nu_{j}^{(e-i)}(N;d_1,\ldots,d_r).\]
\item \label{property-4} Suppose \(1 \leq e\leq r_1 \leq r-1\), then
\[\nu_{j+(r - r_1 - 1)}^{(e)}(N;d_1,\ldots,d_{r_1}) \geq \nu_{j}^{(e)}(N;d_1,\ldots,d_{r}).\]
\end{enumerate}
\end{lemma}

\begin{proof}
The proofs of (\ref{property-1}) and (\ref{property-2}) are
straightforward and omitted.  Clearly, it suffices to prove
(\ref{property-3}) for \(i = 1\).  We first assume
\(1 \leq e \leq r\).  In this case, we have following estimate:
\begin{align*}
  \nu_{j+1}^{(e)}(N;d_1,\ldots,d_r)
  &= j + 1 + \max\left\{ 0, \left\lceil  \frac{N - j -1- \sum_i d_i^{\ast}(e)}{d_1}\right \rceil \right\} \\
  &= j + \max \left\{ 1, \left\lceil \frac{N-j + (d_1 - d_e) - \sum_{i\neq e} d_i^{\ast}(e) - 1}{d_1} \right \rceil \right\} \\
  &\geq j + \max \left\{ 1, \left\lceil \frac{N-j - \sum_i d_i^{\ast}(e-1) }{d_1}\right \rceil \right\} \\
  & \geq \nu_j^{(e-1)}(N;d_1,\ldots,d_r).
\end{align*}

If \(e \geq r+1\), then \(d_{i}(e) = d_{i} = d_{i}(e-1)\), the
property reduces to the inequality
\begin{equation*}
\max\left\{ 0, \frac{N-j-1-\sum_{1}^{r} d_{\nu}}{d_{1}} \right\} + j + 1
\geq
j + \max\left\{ 0, \frac{N-j-\sum_{1}^{r} d_{\nu}}{d_{1}}\right\},
\end{equation*}
which is clearly valid since \(d_{1}\geq 1\).
If \(e \leq 0\), then the property reduces to
\begin{equation*}
\max\left\{
0, \frac{N-j-1- M \cdot d_{1}}{d_{1}}
\right\} + j +1
\geq
\max\left\{
0, \frac{N-j-M\cdot d_{1}}{d_{1}}
\right\}+j
\end{equation*}
where \(M = \operatorname{Card}\{\nu: d_{\nu} = d_{1}\}\).
This is also valid since \(d_{1} \geq 1\).

To prove (\ref{property-4}), we simply compute:
\begin{align*}
  {}& \nu_{j+(r-r_1-1)}^{(e)}(N;d_1,\ldots,d_{r_1}) \\
  = {}& j + (r-r_1-1) + \max\left\{0, \left\lceil \frac{N-j-(r-r_1-1) -\sum_{i=1}^{r_1}d_{i}^{\ast}(e)}{d_1}\right \rceil\right\} \\
  \geq {} & j + \max\left\{ 0, \left\lceil \frac{ N-j + (r-r_1-1)(d_1-1) + \sum_{i=r_1+1}^{r}d^{\ast}_{i}(e) - \sum_{i=1}^{r}d_{i}^{\ast}(e)}{d_1}\right \rceil \right\} \\
  \geq {} & \nu_{j}^{(e)}(N;d_1,\ldots,d_{r})
\end{align*}
The same proof shows that for integer $i\geq 0$ and  \(1 \leq e\leq r_1 \leq r-i\), then
\[\nu_{j+(r - r_1 - i)}^{(e)}(N;d_1,\ldots,d_{r_1}) \geq \nu_{j}^{(e)}(N;d_1,\ldots,d_{r}).\]
But we will only need the case $i=1$ as stated in the lemma.
This completes the proof.
\end{proof}

\subsection*{Conventions}

In the following, \(k\) will either be an algebraic closure finite
field, or the field \(\mathbb{C}\) of complex numbers.  We fix a prime \(\ell\) different from the characteristic of \(k\).

If \(k\) is an algebraic closure of a finite field, and if \(K\) is a
constructible \(\mathbb{Q}_{\ell}\)-complex on a variety \(X\)
(``arithmetic situation''), we will assume that \(K\) is equipped with
a Frobenius action (a complex of ``Weil sheaves'').  This is the case
when \(X = X_{0} \otimes_{\mathbb{F}_{q}} k\) is defined over some
finite subfield \(\mathbb{F}_{q}\), and \(K\) is the inverse image of
some \(\mathbb{Q}_{\ell}\)-complex on \(X_{0}\).



If \(k=\mathbb{C}\) (``Betti situation''), and if \(K\) is a
constructible complex with \(\mathbb{Q}\)-coefficient on
\(X^{\mathrm{an}}\), then we use the symbol
\(\mathrm{H}^{\ast}_{c}(X;K)\) and \(\mathrm{H}^{\ast}(X;K)\)
to denote the cohomology of the complex \(K\) on the analytic space
\(X^{\mathrm{an}}\), with or without compact supports.  We will assume
that \(K\) underlies a mixed Hodge module, so that
\(\mathrm{H}_{c}^{\ast}(X;K)\) and \(\mathrm{H}^{\ast}(X;K)\) are a
mixed Hodge structures.

In either case, the constant sheaf \(\mathbb{Q}_{\ell}\) or
\(\mathbb{Q}\) will be denote by \(\Lambda\), and \(\Lambda(m)\)
denotes the \(m\)\textsuperscript{th} Tate twist.  When we write
\(\mathrm{H}^{i}_{(c)}(X)\), we mean
\(\mathrm{H}^{i}_{(c)}(X;\Lambda)\).

We shall construct and use several maps between cohomology groups in
the main text.  They all come from the standard 6-operations.
Therefore, in the arithmetic situation, these maps are easily shown to
be compatible with the Frobenius actions, once we choose a field of
definition of the variety and the complexes; in the Betti situation,
these maps are maps of Hodge structures.

\section{Generic base change, weak Lefschetz, and a Gysin lemma}

The purpose of this section is to deduce, from Deligne's generic base
change theorem~\ref{theorem:deligne} and weak Lefschetz
theorem~\ref{theorem:weak-lefschetz}, a ``Gysin
lemma''~\ref{lemma:gysin}. This lemma will be the foundation of the
remainder of this paper.

First, recall Deligne's ``generic base change theorem''
(\cite[Finitude, Corollary 2.9]{deligne:sga4.5}).

\begin{theorem}[Deligne]
\label{theorem:generic-base-change}
Let \(S\) be a finite type, separated scheme over \(k\).
Let \(f\colon X \to Y\) be a morphism of finite type, separated \(S\)-schemes.
Let \(K\) be a constructible complex on either \(Y\) or \(X\).
Then there exists a dense open subscheme \(U\) of \(S\), such
that \(Rf_{\ast}K\), \(Rf_{!}K\), \(f^{-1}K\), \(Rf^{!}K\) are all constructible complexes
on \(X\) or \(Y\), and the formation is compatible with
all base change \(S^{\prime} \to U \subset S\).
\end{theorem}

From the generic base change theorem we deduce the following ``generic purity lemma''.

\begin{lemma}
\label{lemma:generic-purity}
Let \(f\colon X \to \mathbb{P}^{N}\) be a quasi-finite morphism.
Let \(K\) be a constructible complex of \(\Lambda\)-modules on $X$.
For a hyperplane \(B\),
form the fiber diagram
\begin{equation*}
\begin{tikzcd}
X\times_{\mathbb{P}^{N}} B \ar[r,"\iota^{\prime}_B"] \ar[d,"f^{\prime}"] & X \ar[d,"f"] \\
B \ar[r,"\iota_B"] & \mathbb{P}^N
\end{tikzcd}.
\end{equation*}
If \(B\) is sufficiently general,
then the purity morphism
\(K|_{X\times_{\mathbb{P}^{N}} B}[-2](-1) \xrightarrow{\sim} R(\iota_{B}^{\prime})^{!}K\)
is an isomorphism.
\end{lemma}

\begin{proof}
First, assume that \(f\) is a finite morphism.  Let \(\mathcal{M}\) be
the space of hyperplanes in \(\mathbb{P}^N\) (i.e., \(\mathcal{M}\) is
the dual projective space \(\check{\mathbb{P}}^N\)).  Form the
incidence correspondence
\(\mathcal{A} = \{(x,[B]) \in \mathbb{P}^N \times \mathcal{M}: x \in
B\}\).  Let
\(\iota\colon \mathcal{A} \to \mathbb{P}^N_{\mathcal{M}} =
\mathbb{P}^N \times \mathcal{M}\) be the closed immersion.  Then the
projections
\(\mathrm{pr}_1\colon\mathbb{P}^N_{\mathcal{M}} \to \mathbb{P}^N\),
and \(\mathrm{pr}_{1}\circ\iota\colon\mathcal{A} \to \mathbb{P}^N\)
are all smooth, equidimensional morphisms. Let
\(L = \mathrm{pr}_{1}^{-1}(Rf_{\ast}K)\).  Since \(L\) is the inverse
image of a complex by a smooth, equidimensional morphism, and since
\(\mathrm{pr}_{1} \circ \iota\) is also smooth and equidimensional,
the relative purity theorem implies that the purity morphism is an
isomorphism
\begin{equation}
  \label{eq:purity-iso}
  \iota^{-1}L \xrightarrow{\sim} R\iota^{!}L[2](1).
\end{equation}
See, for example,
\cite[Theorem~11.2, Supplement]{kiehl-weissauer:weil-conjecture-perverse-sheaf-fourier-transform}.

By Theorem \ref{theorem:generic-base-change},
there is a Zariski
dense open subset \(\mathcal{U}\) of \(\mathcal{M}\) such that the formation of \(R\iota^{!}L\) commutes with arbitrary
base change \(S \to \mathcal{U}\). In particular, for any \([B] \in \mathcal{U}\),
we have
\begin{equation}\label{eq:generic-base-change-special-case}
  (R\iota^{!}L)|_{B} \simeq R\iota_B^{!}(L|_{\mathbb{P}^{N}\times\{[B]\}}).
\end{equation}
Here we are using the following base change diagram
\begin{equation*}
  \begin{tikzcd}
    B \ar[r,"\iota_B"] \ar[hook,d] & \mathbb{P}^N\times\{[B]\} \ar[r] \ar[hook,d] & \{[B]\} \ar[hook,d] \\
    \mathcal{A} \ar[r,"\iota"] & \mathbb{P}^N_{\mathcal{M}} \ar[r,"\mathrm{pr}_{2}"] & \mathcal{M}
  \end{tikzcd}.
\end{equation*}

Then \eqref{eq:generic-base-change-special-case} implies that
\((R\iota^{!}L)|_{B} \simeq R\iota_{B}^{!}Rf_{\ast}K\).
Coupled with the purity isomorphism \eqref{eq:purity-iso}, we get
\begin{equation*}
  \iota_{B}^{-1}Rf_{\ast}K[-2](-1) \simeq R\iota_{B}^{!} Rf_{\ast}K.
\end{equation*}
Now we turn to the following fiber diagram:
\begin{equation*}
  \begin{tikzcd}
    B \times_{\mathbb{P}^{N}} X \ar[d,"f^{\prime}"] \ar[r,"\iota^{\prime}_B"] & X \ar[d,"f"] \\
    B \ar[r,"\iota_B"] & \mathbb{P}^N
  \end{tikzcd}.
\end{equation*}
By proper base change, we see \(\iota_{B}^{-1}Rf_{\ast}K \simeq Rf^{\prime}_{\ast}(\iota^{\prime}_{B})^{-1}K\).
Applying Verdier duality to the above isomorphism gives
\(R\iota_B^{!} Rf_{!} \mathbb{D}K \simeq Rf^{\prime}_{!}R(\iota^{\prime}_{B})^{!}\mathbb{D}K\).
Here \(\mathbb{D}K\) is the Verdier dual of \(K\).
Replacing \(K\) by \(\mathbb{D}K\) in this latter isomorphism,
using \(Rf_{!} = Rf_{\ast}\) (as \(f\) is proper)
and \(Rf^{\prime}_{!} = Rf^{\prime}_{\ast}\), we conclude that
\begin{equation}\label{eq:generic-pre-purity}
  Rf^{\prime}_{\ast}(\iota^{\prime}_{B})^{-1}K[-2](-1) \simeq Rf^{\prime}_{\ast}R(\iota^{\prime}_{B})^{!}K.
\end{equation}
Since \(f\) is a finite morphism, so is \(f^{\prime}\).  In this situation, we
have \(Rf^{\prime}_{\ast} =f^{\prime}_{\ast}\).  For any closed point \(b\) of \(B\),
the stalks of both sides of \eqref{eq:generic-pre-purity} at \(b\) are
the direct sums of the stalks of \((\iota^{\prime}_{B})^{-1}K[-2](-1)\) and
\(R(\iota^{\prime}_{B})^{!}K\) at the points in \(f^{-1}(b)\).  Since the
purity map induces an isomorphism on the direct sum of these stalks,
the purity map itself is necessarily an isomorphism.
Therefore, \eqref{eq:generic-pre-purity} implies that
\(K|_{B \times_{\mathbb{P}^{N}} X}[-2](-1) \simeq R(\iota^{\prime}_{B})^{!}K\).
This finishes the proof when \(f\) is a finite morphism.

Now assume \(f\) is a quasi-finite morphism.
Then we can factor \(f\) into a composition
\(X \xrightarrow{j} \overline{X} \xrightarrow{g} \mathbb{P}^{N}\) in which
\(j\) is an open immersion, \(g\) is a finite morphism.
Thus we have the following fiber diagram
\begin{equation*}
  \begin{tikzcd}
    X\times_{\mathbb{P}^{N}} B \ar[d,"j^{\prime}"] \ar[r,"\iota^{\prime}_B"] & X \ar[d,"j"]  \\
    \overline{X} \times_{\mathbb{P}^{N}} B \ar[d,"g^{\prime}"] \ar[r,"\iota^{\prime\prime}_B"] & \overline{X} \ar[d,"g"] \\
    B \ar[r,"\iota_B"] & \mathbb{P}^N
  \end{tikzcd}.
\end{equation*}
In the diagram, \(f^{\prime} = g^{\prime} \circ j^{\prime}\).
By the previous paragraph, for a sufficiently
general \(B\), we have
\[
(j_{\ast}K)|_{\overline{X} \times_{\mathbb{P}^{N}} B}[-2](-1)\simeq R(\iota_{B}^{\prime\prime})^{!}(j_{\ast}K).
\]
Applying the functor \((j^{\prime})^{-1} = R(j^{\prime})^{!}\) to the above
isomorphism gives the desired isomorphism
\begin{equation*}
K|_{X\times_{\mathbb{P}^{N}} B} [-2](-1) \simeq R(\iota_{B}^{\prime})^{!}K.
\end{equation*}
This completes the proof.
\end{proof}

Next, let us recall Deligne's weak Lefschetz theorem for perverse sheaves
\cite[Corollary~A.5]{katz:affine-cohomological-transforms-perversity-monodromy}.

\begin{theorem}[Deligne]
\label{theorem:weak-lefschetz}
Let \(f\colon X \to \mathbb{P}^{N}\) be a quasi-finite morphism to a projective space.
Let \(K\) be a perverse sheaf on \(X\). Then for a sufficiently general
hyperplane \(B\), the restriction morphism
\begin{equation*}
\mathrm{H}^i(X;K) \to \mathrm{H}^i(f^{-1}B; K|_{f^{-1}B})
\end{equation*}
is injective if \(i = -1\), and bijective if \(i < -1\).
\end{theorem}

\begin{remark}%
  Going through the proof of Theorem~\ref{theorem:weak-lefschetz}, one
  finds that in the statement one may replace the perversity of \(K\)
  by \(K\) satisfying the \emph{cosupport condition}.  Recall that a
  constructible complex \(K\) is said to be \emph{semiperverse}, or to
  satisfy the \emph{support condition}, if
  \(\dim\operatorname{Supp}\mathcal{H}^{i}(K)\leq-i\).  In terms of
  the perverse \(t\)-structure, this means that \(K\) has vanishing
  positive perverse cohomology sheaves.  If \(\mathbb{D}K\) satisfies
  the support condition then we say \(K\) satisfies the
  \emph{cosupport condition}.
\end{remark}

\begin{lemma}%
\label{lemma:gysin-perv}
Let \(f\colon X \to \mathbb{P}^N\) be a quasi-finite morphism.  Let
\(K\) be a perverse sheaf on \(X\) (or more generally, \(K\) satisfies
the \emph{support condition}).  Then for a sufficiently general
hyperplane, and any \(i\), there exists a Gysin map
\begin{equation*}
\mathrm{H}^{i-2}_{c}(X \times_{\mathbb{P}^{N}} B;K|_{X\times_{\mathbb{P}^{N}} B}(-1)) \to \mathrm{H}^i_{c}(X;K),
\end{equation*}
which is surjective if \(i = 1\), and is bijective if \(i \geq 2\).
\end{lemma}

\begin{proof}
  We choose a sufficiently generic hyperplane \(B\) such that the
  conclusion of Theorem~\ref{theorem:weak-lefschetz} holds for
  \(\mathbb{D}K\), and such that the conclusion of
  Lemma~\ref{lemma:generic-purity} holds for \(K\).  Thus
  \begin{equation*}
    \mathrm{H}^i(X;\mathbb{D}K) \to \mathrm{H}^i(X\times_{\mathbb{P}^{N}} B; \iota_B^{\ast}(\mathbb{D}K))
  \end{equation*}
  is injective for \(i = -1\), and is bijective for \(i < -1\).
  Taking Verdier duality yields that the natural map
  \begin{equation*}
    \mathrm{H}^i_{c}(X\times_{\mathbb{P}^{N}} B; R\iota_B^{!}K) \to  \mathrm{H}^i_{c}(X;K)
  \end{equation*}
  is surjective for \(i = 1\), and is bijective for \(i > 1\).
  By our choice, \(R\iota_B^{!}K \simeq \iota_B^{\ast}K[-2](-1)\).
  Therefore we obtain a map
  \begin{equation*}
    \mathrm{H}^{i-2}_{c}(X\times_{\mathbb{P}^{N}} B; K|_{X\times_{\mathbb{P}^{N}} B}(-1)) \to  \mathrm{H}^i_{c}(X;K),
  \end{equation*}
  which is surjective for \(i = 1\), and bijective for \(i > 1\).
\end{proof}

Let \(X\) be an algebraic variety of dimension \(\leq n\).  Then for
any constructible \(\Lambda\)-sheaf \(\mathcal{F}\) on \(X\), it is
trivial to show that \(\mathcal{F}[n]\) satisfies the support condition.
Applying Lemma~\ref{lemma:gysin-perv} gives the following result.

\begin{lemma}
\label{lemma:gysin}
Let \(X\) be a \(k\)-variety of dimension \(\leq n\).  Let
\(f\colon X \to \mathbb{P}^{N}\) be a quasi-finite immersion.  Let
\(\mathcal{F}\) be a constructible sheaf on \(X\).  Then for a
sufficiently general hyperplane \(B\), there exists a Gysin map
\begin{equation*}
\mathrm{H}^{i-2}_{c}(X \times_{\mathbb{P}^{N}} B ; \mathcal{F}(-1)) \to \mathrm{H}^{i}_{c}(X ; \mathcal{F})
\end{equation*}
which is surjective if \(i = n + 1\), and is bijective if
\(i \geq n + 2\).\qed
\end{lemma}

\section{The method of del~Angel and Esnault--Katz}

\begin{situation}\label{situation:overall}
  Let \(f_1,\ldots,f_r \in k[x_1,\ldots,x_N]\) be a collection
  of polynomials.  Let \(X\) be the common zero locus of
  \(f_1,\ldots,f_r\) in \(\mathbb{A}^N\).
  Assume that \(d_1 \geq \cdots \geq d_r\) is a sequence of numbers
  such that \(d_i \geq \deg f_i\).
  Let \(n = \dim X\).  In general, \(n \geq N-r\); the equality holds
  if and only if \(X\) is a (set-theoretic) complete intersection by the \(r\) defining equations.
\end{situation}

The following lemma proves the main theorems when \(X\) is a
set-theoretic complete intersection, and it also proves
Theorem~\ref{theorem:positive}.

\begin{lemma}%
  \label{lemma:complete-intersection-bound}
  In Situation~\ref{situation:overall}, the colevel of
  \(\mathrm{H}^{n+j}_{c}(X)\) is
  \(\geq \mu_j(N;d_1,\ldots,d_r)\) for $j\geq 0$.
\end{lemma}

\begin{proof}
  Choose \(n\) sufficiently general hyperplanes \(B_1,\ldots,B_{n}\).
  Since \(B_1,\ldots,B_{n}\) are general, for each \(j\),
  \(X \cap B_1 \cap \cdots \cap B_j\) is a Zariski closed subset cut
  out in \(\mathbb{A}^{N-j}\) by equations of degrees at most
  \(d_1,\ldots,d_r\), and has dimension equal to \(n-j\).

  Applying the Gysin lemma, Lemma~\ref{lemma:gysin}, to \(X_i\)
  gives the following chain of maps:
  \begin{equation*}
    \begin{tikzcd}[column sep=tiny]
      \mathrm{H}^{n}_{c}(X) & \mathrm{H}^{n+1}_{c}(X) & \mathrm{H}^{n+2}_{c}(X) & \cdots \\
      & \mathrm{H}^{n-1}_{c}(X \cap B_1;\Lambda(-1)) \ar[two heads,u] & \mathrm{H}^{n}_{c}(X \cap B_1;\Lambda(-1)) \ar[u, "\sim" {anchor=south, rotate=90}] &\cdots\\
      & & \mathrm{H}^{n-2}_{c}(X \cap B_1 \cap B_2;\Lambda(-2)) \ar[u,two heads] & \cdots \\
      & & & \ddots
    \end{tikzcd}
  \end{equation*}
  It follows that the colevel of \(\mathrm{H}^{n+j}_{c}(X)\) is at least
  the colevel of \(\mathrm{H}^{n-j}_{c}(X \cap B_1 \cap \cdots \cap B_j)\) plus \(j\).
  By the theorem of Esnault and
  Katz, Theorem~\ref{theorem:previous}(\ref{item-1a}),
  the colevel of \(\mathrm{H}^{n-j}_{c}(X \cap B_1 \cap \cdots \cap B_j)\) is a least \(\mu_0(N-j;d_1,\ldots,d_r)\).
  It follows that the colevel of \(\mathrm{H}^{n+j}_{c}(X)\) is at least
  \(\mu_0(N-j;d_1,\ldots,d_r)+j = \mu_j(N;d_1,\ldots,d_r)\).
\end{proof}

While Lemma~\ref{lemma:complete-intersection-bound} answers positively
Question~\ref{question}, it does not provide the most optimal bounds
unless \(n = N-r\), and it does not provide information for colevels
of \(\mathrm{H}^j_{c}\) when \(N-r\leq j \leq n\).  In the following, we
shall focus on proving the improved bounds stated in
Theorems~\ref{theorem:beyond} and \ref{theorem:before}.

We shall first treat the case when \(X\) is ``not too far from'' a
complete intersection; that is, \(X\) is obtained from a complete
intersection by adding an extra equation which does not drop the
dimension.  See
Lemma~\ref{lemma:almost-complete-intersection-colevel}.  Note however
we formulate the result in a slightly more general way.  This is to
facilitate an induction argument used in next section.  This sort of
formulation had been used by del
Angel~\cite{del-angel:hodge-type-projective-varieties-low-degree} and
Esnault--Katz.  Let us state the theorem of
Esnault--Katz~\cite[Theorem~2.1,
\S5.3]{esnault-katz:cohomological-divisibility}, which is slightly
stronger than Theorem~\ref{theorem:previous}(\ref{item-1a}), and will
be used later.

\begin{theorem}[del~Angel, Esnault--Katz]%
\label{ek}
Suppose we are given a list of polynomials
\(f_1, \ldots ,f_r \in k[x_1, \ldots, x_N]\) with
\(\deg f_i \leq d_i\), \(d_i \geq 1\).  Take any integer \(R \geq 1\)
and any list of polynomials
\(g_1,\ldots,g_R \in k[x_1, \ldots, x_N]\), such that each \(g_j\) is
a product \(g_j = \prod_{i}f_i^{a_{i,j}}\) with \(a_{i,j} \geq 0\).
Denote by \(U\) the closed subscheme of \(\mathbb{A}^N\) defined by
\(g_1 = \cdots = g_R = 0\).  Then for any \(m \in \mathbb{Z}\), the
colevel of \(\mathrm{H}^{m}_{c}(U)\) is at least
\(\mu_0(N; d_1, \ldots, d_r)\).
\end{theorem}

\begin{situation}
  \label{hypothesis}
  In Situation~\ref{situation:overall},
  let \(V\) be the variety defined by \(f_1 = \cdots = f_{N-n} = 0\).
  We make the following additional hypothesis:
  \[
    \dim V = \dim X = n.
  \]
  In other words, we assume in addition that \(V\) is a set-theoretic
  complete intersection cut out by \(f_{1},\ldots,f_{N-n}\).

  For every nonempty subset \(J\)
  of \(\{N-n+1,\ldots,r\}\), let \(X_J\) be the closed
  subscheme of \(V\) defined by \(\prod_{j\in J}f_j = 0\).  Since
  \(V \supset X_J \supset X\), we know \(\dim X_J = n\).
  Note that when $J$ is the empty set, we have $X_J=V$ by convention.
\end{situation}

\begin{lemma}%
  \label{lemma:basis-induction-almost-complete-intersection-colevel}
  In Situation~\ref{hypothesis}, for every \(J \subset \{N-n+1,\ldots,r\}\),
  the colevel of \(\mathrm{H}^{n}_{c}(V\setminus X_J)\) is
  \(\geq \mu_0(N;d_1,\ldots,d_r)\).
\end{lemma}

\begin{proof}
  Since \(V\) is an affine complete intersection of pure dimension \(n\),
  and since \(X_J\) is a hypersurface in \(V\), the open subscheme
  \(V\setminus X_J\) is an affine local complete intersection of pure dimension \(n\).
  Consider the following long exact sequence
  \begin{equation}\label{eq:les}
    \begin{tikzcd}
      0 \arrow[r]
      & 0 \arrow[r]
      \arrow[d, phantom, ""{coordinate, name=F}]
      & \mathrm{H}^{n-1}_{c}(X_J)
      \arrow[dll,
      rounded corners,
      to path={ -- ([xshift=2ex]\tikztostart.east)
        |- (F) [near end]\tikztonodes
        -| ([xshift=-2ex]\tikztotarget.west)
        -- (\tikztotarget)}] \\
      \mathrm{H}^{n}_{c}(V\setminus X_J) \arrow[r]
      & \mathrm{H}^{n}_{c}(V) \arrow[r]
      \arrow[d, phantom, ""{coordinate, name=E}]
      & \mathrm{H}^{n}_{c}(X_J) \arrow[dll, rounded corners,
      to path={ -- ([xshift=2ex]\tikztostart.east)
        |- (E) [near end]\tikztonodes
        -| ([xshift=-2ex]\tikztotarget.west)
        -- (\tikztotarget)}] \\
      \mathrm{H}^{n+1}_{c}(V\setminus X_J) \arrow[r]
      & \mathrm{H}^{n+1}_{c}(V) \arrow[r]
      & \mathrm{H}^{n+1}_{c}(X_J) \arrow[r] & \cdots.
    \end{tikzcd}
  \end{equation}

  By Lemma \ref{lemma:complete-intersection-bound}, the colevel of
  \(\mathrm{H}^{n}_{c}(V)\) is at least
  \(\mu_0(N;d_1,\ldots,d_{N-n})\).
  By Theorem~\ref{ek}
  (\(R=N-n+1\), \(U = X_J = \{f_1=\cdots=f_{N-n}=\prod_{j\in J}f_j=0\}\)),
  the colevel of \(\mathrm{H}^{n-1}_{c}(X_J)\) is at least
  \(\mu_0(N;d_1,\ldots,d_r)\).
  From the  exact sequence~(\ref{eq:les}), we see the colevel of
  \(\mathrm{H}^{n}_{c}(V\setminus X)\) is at least
  \(\mu_0(N;d_1,\ldots,d_r)\).
\end{proof}




\begin{lemma}%
  \label{lemma:almost-complete-intersection-colevel} Assume $r\geq N-n+1$.
  In Situation~\ref{hypothesis}, for every subset \(J\) of \(\{N-n+1,\ldots,r\}\),
  the colevel of \(\mathrm{H}^{n+j}_{c}(X_J)\) is at least \(\nu_j^{(r-1)}(N;d_1,\ldots,d_r)\).
\end{lemma}

\begin{proof}
  From the exact sequence \eqref{eq:les} we see that the colevel of
  \(\mathrm{H}^{n+j}_{c}(X_J)\) is controlled by
  \(\mathrm{H}^{n+j}_{c}(V)\) and
  \(\mathrm{H}^{n+j+1}_{c}(V\setminus X_J)\).
  The colevel of the former is at least \(\mu_j(N;d_1,\ldots,d_{N-n})\),
  as shown in Lemma~\ref{lemma:complete-intersection-bound}.
  Clearly
  \begin{equation*}
    \frac{N-j - \sum_{i=1}^{N-n}d_i}{d_1} \geq \frac{N-j - \sum_{i=1}^{r}d_i^{\ast}(e)}{d_1},
  \end{equation*}
  for every \(e \geq N-n\).  Taking \(e = N-n \leq r-1\),
  we conclude that
  \[\mu_j(N;d_1,\ldots, d_{N-n}) \geq \nu_j^{(N-n)}(N;d_1,\ldots,d_r) \geq \nu_j^{(r-1)}(N;d_1,\ldots,d_r).\]

  It remains to prove the colevel of \(\mathrm{H}^{n+j+1}_{c}(V\setminus X_J)\) is
  \(\geq \nu_{j}^{(r-1)}(N;d_1,\ldots,d_r)\).

  For each \(j\), choose a sufficiently general linear subspace \(B\) in
  \(\mathbb{A}^N\) of codimension \(j+1\).
  The Gysin lemma, Lemma~\ref{lemma:gysin}, applied to the locally closed immersion
  \[
    V\setminus X_J \hookrightarrow \mathbb{A}^{N} \subset \mathbb{P}^N
  \] (\(j+1\) times), gives a surjection
  \[
    \mathrm{H}_{c}^{n-(j+1)}((V \cap B) \setminus (X_{J}\cap B);\Lambda(-j-1)) \to
    \mathrm{H}_{c}^{n+j+1}(V \setminus X_J).
  \]
  Now, $V \cap B$ is a complete intersection of dimension $n-(j+1)$ in $\mathbb{A}^{N-(j+1)}$.
  By Lemma~\ref{lemma:basis-induction-almost-complete-intersection-colevel}, we deduce that
  the colevel of \(\mathrm{H}^{n+j+1}_{c}(V\setminus X_J)\)
  is at least
  \[
    j+1 + \mu_0(N-j-1;d_1,\ldots,d_r)=\mu_{j+1}(N;d_1,\ldots,d_r).
  \]
  Since \(r \geq N-n+1 > 1\), Lemma~\ref{lemma:properties-of-nu}(\ref{property-3}) can be
  applied, and we see the above number is
  \(\geq \nu_j^{(r-1)}(N;d_1,\ldots,d_r)\).
\end{proof}

\section{Proofs of main theorems}

In this section we prove Theorems~\ref{theorem:beyond} and
\ref{theorem:before}.  It suffices to prove the colevel bounds for the
affine variety $X$.  The colevel bounds for the complement
$\mathbb{A}^N\setminus X$ then follows by excision.

We will need the following simple lemma.

\begin{lemma}%
  \label{lemma:spectral-sequence}
  Let \(W\) be an algebraic variety.
  Let \(W_1,\ldots,W_{s}\) be a collection of Zariski closed subsets of \(W\),
  such that \(W = \bigcup_{\nu=1}^s W_{\nu}\).
  Then the colevel of \(\mathrm{H}^i_{c}(W_1 \cap \cdots \cap W_s)\) is at least the
  minimum of the colevels of the following cohomology groups:
  \begin{align*}
    \bigoplus_{|I|=s-a} \mathrm{H}^{i+a-1}_{c}\left(\bigcap_{\nu\in I}W_{\nu}\right),\quad a=1,2,\ldots,s.
  \end{align*}
\end{lemma}

We remark that when \(a = s\), the intersection
\(\bigcap_{\nu\in I}W_{\nu}\), $|I|=s-a$, is an empty intersection,
which corresponds to the entire set \(W\).  So the boundary case $a=s$
of the statement corresponds to \(\mathrm{H}^{i+s-1}_{c}(W)\).

\begin{proof}
There is a first-quadrant Mayer--Vietoris spectral sequence
\begin{equation*}
E_1^{a,b} = \bigoplus_{|I|=a+1}\mathrm{H}^b_{c}\left(\bigcap_{\nu\in I}W_{\nu}\right) \Rightarrow \mathrm{H}_{c}^{a+b}(W)
\end{equation*}
associated to the finite proper hypercovering \(W = \bigcup W_{\nu}\).
In the \(\ell\)-adic context, the differentials preserve the
geometric Frobenius operation; in the Betti context, the
differentials are morphisms of mixed Hodge structures.

To explain the idea of the proof, we have drawn below a schematic
diagram, Figure~\ref{figure:spectral}. In the figure, the gray squares
represent the positions in the spectral sequence that may hit the
terms sitting in the red square, which are $E^{s-1,j}_{\bullet}$.  A
certain subquotients of the first-page items sitting in the gray
squares will span a linear subspace of \(E^{s-1,j}_{1}\).  Dividing
by this subspace gives a quotient space of \(E_{1}^{s-1,j}\), which
equals \(E_{\infty}^{s-1,j}\) and is also a quotient of the cohomology
of \(W\).  Therefore, to estimate the lower bound of the colevel of
the first-page item at the red square, that is
\(E_{1}^{s-1,\ast}=\mathrm{H}_{c}^{\ast}(\bigcap_{\nu=1}^{s}W_{\nu})\),
we only need some lower bounds of the colevels of the first-page items
sitting in the gray squares, and a lower bound of the colevel of
\(\mathrm{H}^{s-1+\ast}_{c}(W)\).

\begin{figure}[ht!]
\centering
\begin{tikzpicture}
[
scale=0.8,
box/.style={rectangle,draw=gray,thick, minimum size=.8cm},
]

\foreach \x in {1,2,...,5}{
  \foreach \y in {0,1,...,5}
  \node[box] at (\x,\y){};
}
\node[box,fill=red!60] at (5,2){};
\node[box,fill=gray!60] at (4,2){};
\node[box,fill=gray!60] at (3,3){};
\node[box,fill=gray!60] at (2,4){};
\node[box,fill=gray!60] at (1,5){};
\node at (6.2,2) {$E^{s-1,j}_{\bullet}$};
\node [gray] at (0,5) {$E^{s-5,j+3}_{\bullet}$};
\node [gray] at (1,4) {$E^{s-4,j+2}_{\bullet}$};
\draw[->] (1,5.2) -- (5,2.3) node[midway,above,sloped] {\footnotesize \(d_4\)};
\draw[->] (2,4.2) -- (5,2.2) node[midway,below,sloped] {\footnotesize \(d_3\)};
\draw[->] (3,2.8) -- (5,2.1) node[midway,below,sloped] {\footnotesize \(d_2\)};
\draw[->] (4,2) -- (5,2) node[midway,below] {\footnotesize \(d_1\)};
\end{tikzpicture}
\caption{The spectral sequence}
\label{figure:spectral}
\end{figure}
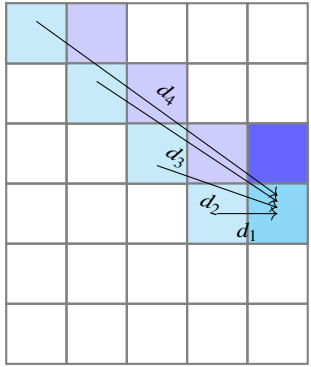

Now we give the formal proof.  The cohomology group
\(\mathrm{H}^{\ast}_{c}(W_1 \cap \cdots \cap W_s)\) in
question equals \(E_1^{s-1,\ast}\) (indicated by the red box).  As it
is in the right-most column of the spectral sequence, all the
differentials out of it are zero.  It follows that
\(E_{\infty}^{s-1,\ast}\) is a quotient of \(E_{1}^{s-1,\ast}\).  We
have
\begin{equation*}
E_2^{s-1,\ast} = E_1^{s-1,\ast}/d_1(E_1^{s-2,\ast}),\;
\ldots,\; E_{a+1}^{s-1,\ast} = E_{a}^{s-1,\ast}/d_{a}(E_a^{s-a-1,\ast+a-1}), \; \ldots,
\end{equation*}
and
\(E_{\infty}^{s-1,\ast} = E_{s}^{s-1,\ast}\).
Since each item \(E_{a}^{s-a-1,\ast+a-1}\) is a subquotient of
\[E_1^{s-a-1,\ast+a-1}=\bigoplus_{|I|=s-a}\mathrm{H}^{\ast+a-1}\left(\bigcap_{\nu\in I}W_{\nu}\right)\]
(indicated by the gray boxes),
it follows that the colevel of the kernel of \(E^{s-1,\ast}_{1} \to E_{\infty}^{s-1,\ast}\)
is at least the minimum of
\begin{equation*}
\bigoplus_{|I|=s-a} \mathrm{H}^{\ast+a-1}_{c}\left(\bigcap_{\nu\in I}W_{\nu}\right),\quad a=1,2,\ldots,s-1.
\end{equation*}
To complete the proof, it remains to bound the colevel of \(E_{\infty}^{s-1,\ast}\).
Since \(E_{\infty}^{s-1,\ast}\) is a subspace of
the abutment \(\mathrm{H}^{\ast+s-1}_{c}(W)\) (the
(\(s-1\))\textsuperscript{st} column is the right-most column), its colevel
is at least the colevel of \(\mathrm{H}^{\ast+s-1}_{c}(W)\).
This completes the proof.
\end{proof}

We also need the following crucial algebraic lemma.  It allows us to
rebase the problem from \(\mathbb{A}^N\) to a set-theoretic complete
intersection of dimension \(n\).  Since the proof is a bit lengthy, we
shall not reproduce it here. The reader is referred to
\cite[Lemma~5.1]{wz22}.

\begin{lemma}%
  \label{lemma:key}
  Let \(k\) be a field.
  Let \(f_1,\ldots,f_r \in k[x_1,\ldots,x_N]\) be a collection
  of polynomials. Assume that \(\deg f_1 \geq \cdots \geq \deg f_r\).
  Let \(X\) be the affine variety defined by
  \(f_1 = \cdots = f_r=0\) in \(\mathbb{A}^{N}\).
  Then, upon replacing \(k\) by a finite extension, there exists
  a new sequence of polynomials \(g_1, \ldots, g_r \in k[x_1,\ldots,x_N]\)
  such that
  \begin{enumerate}
  \item \(X = \{g_1 = \cdots = g_r = 0\}\),
  \item \(\deg g_i \leq \deg f_i\),
  \item the dimension of the variety \(V\) defined by \(\{g_1 = \cdots = g_{N-\dim X}=0\}\)
    is \(\dim X\).
  \end{enumerate}
\end{lemma}

\begin{proof}[Proof of Theorem \ref{theorem:beyond}]
Recall the definition of the numbers
\(\nu_{j}^{(e)}(N;d_{1},\ldots,d_{r})\) in
Definition~\ref{definition:nu}.  By Lemma~\ref{lemma:key}, we can
replace \(f_{i}\) by \(g_{i}\) without changing \(X\). Since
\(\nu^{(e)}_{j}\) is decreasing in \(d_{j}\)
(Lemma~\ref{lemma:properties-of-nu}(\ref{property-2})), the inequality
``colevel \(\geq \nu_{j}^{(e)}(\deg g_{1},\ldots,\deg g_{r})\)'' is
stronger than the desired inequality.  For this reason, we may replace
\(f_{i}\) by \(g_{i}\), and it suffices prove the theorem under the
following hypothesis:
\begin{equation}
\label{eq:good-case}
\text{The variety \(V\) defined by \(f_1 = \cdots  = f_{N-n} = 0\)
  has dimension \(n\).}
\end{equation}
Under hypothesis~(\ref{eq:good-case}), \(X\) is a subvariety of \(V\)
defined by \(r-(N-n)\) ``extra equations''; these equations do not
drop the dimension: \(\dim X = \dim V\).

We shall prove the theorem by induction on \(r-(N-n)\), the number of
extra equations needed to cut out \(X\) in \(V\).  When this number is
\(0\), the theorem is Lemma~\ref{lemma:complete-intersection-bound}.
Assume now there exist \(r-(N-n)>0\) ``extra equations''.  The theorem
is already true if $r-(N-n)=1$ by
Lemma~\ref{lemma:almost-complete-intersection-colevel}.  But, we will
need the more general form of
Lemma~\ref{lemma:almost-complete-intersection-colevel} in the
induction below.

For \(i \in \{N-n+1,\ldots,r\}\), let \(X_i = V \cap \{f_{i}=0\}\).
Then \(X = X_{N-n+1} \cap \cdots \cap X_{r}\).  In view of
Lemma~\ref{lemma:spectral-sequence} (with \(s=r-(N-n)\)), it suffices
to prove
\begin{quote}
\textit{for each \(I \subset \{N-n+1,\ldots,r\}\)
  such that \(|I| = r-(N-n)-a\), \(1 \leq a\leq r-(N-n)\), the colevel of
  \(\mathrm{H}^{n+j+a-1}_{c}\left( \bigcap_{i \in I} X_i \right)\)
  is at least \(\nu^{(N-n)}_j(N;d_1,\ldots,d_r)\).}
\end{quote}

\medskip
\textbf{\textit{Situation I.}} \(1 \leq a \leq r-(N-n)-1\).

This case is handled by the inductive hypothesis.  The intersection
\(\bigcap_{i\in I}X_i\) is a subvariety of \(V\) of dimension \(n\),
which is cut out by \(r-(N-n)-a<r-(N-n)\) equations.  The
inductive hypothesis implies that
\(\mathrm{H}^{n+j+a-1}_{c}(\bigcap_{i\in I} X_i)\) has
colevel \(\geq \nu^{(N-n)}_{j+a-1}(d_1,\ldots,d_{r-a})\).  By
Lemma~\ref{lemma:properties-of-nu}(\ref{property-4}), this number is
\(\geq \nu_{j}^{(N-n)}(N;d_1,\ldots,d_r)\).

\medskip
\textbf{\textit{Situation II.}} \(a = r-(N-n)\).
This is the case when the intersection is empty.
Let \(W = X_{N-n+1} \cup \cdots \cup X_{r}\).
We need to show the colevel of
\(\mathrm{H}^{n+j+r-(N-n)-1}_{c}(W)\) is at least
\[\nu_{j}^{(N-n)}(N;d_1,\ldots,d_r).\]
This is handled by
Lemma~\ref{lemma:almost-complete-intersection-colevel}: The variety
\(W\) is defined by the equation \(f_{N-n+1}\cdots f_r = 0\), and
satisfies the hypothesis of the lemma.  Hence the colevel of
\(\mathrm{H}_{c}^{n+j+r-(N-n)-1}(X)\) is at least
\[
\nu_{j+r-(N-n)-1}^{(r-1)}(N;d_1,\ldots,d_r).
\]
Since $r> 0$, we have $N-n > 0$. It follows that $r-1 > r - (N-n)- 1$.
We can then apply
Lemma~\ref{lemma:properties-of-nu}(\ref{property-3}), and conclude
that
\begin{align*}
  \nu_{j+r-(N-n)-1}^{(r-1)}(N;d_1,\ldots,d_r)
  & \geq \nu_{j}^{(r-1-(r-(N-n)-1))}(N;d_1,\ldots,d_r)\\
  & =\nu_j^{(N-n)}(N;d_1,\ldots,d_r).
\end{align*}
This completes the proof.
\end{proof}

\begin{proof}[Proof of Theorem \ref{theorem:before}]
As before, without loss of generality, we may assume the additional
hypothesis~(\ref{eq:good-case}) holds.  We shall prove by induction
on \(r-(N-n)\), the number of ``extra equations'' that are needed to
cut out \(X\) from the complete intersection $V$.

The base case is when \(N-n = r\), i.e., when \(X=V\)
is a set-theoretic complete intersection.
In this case, Theorem~\ref{theorem:before} falls back
to Lemma~\ref{lemma:complete-intersection-bound}.

Now assume \(r > N-n\).  If \(m=0\), then we are reduced to the
theorem of Esnault and Katz
(Theorem~\ref{theorem:previous}(\ref{item-1a})).  When \(m > 0\), by
Lemma~\ref{lemma:spectral-sequence} (\(s=r-(N-n)\)), it suffices to
show, that
\begin{quote}
\emph{for each subset \(I\) of \(\{N-n+1,\ldots,r\}\) such that
  \(|I| = r-(N-n)-a\) for some \(1 \leq a \leq r-(N-n)\),
  and each integer \(m\) such that \(0\leq m \leq r-(N-n)\),
  the colevel of \(\mathrm{H}^{N-r+m+a-1}_{c}(\bigcap_{i\in I} X_i)\)
  is at least \(\nu_0^{(r-m)}(N;d_1\ldots,d_r)\).}
\end{quote}

Again, this is an elementary consequence of
Lemma~\ref{lemma:properties-of-nu}.

\medskip%
\textbf{\textit{Situation I.}} \(1 \leq a \leq r-(N-n)-1\).

There are two subcases:
\begin{enumerate}
\item   \(N-r+m+a-1 < n\), i.e.,
when the cohomological degree in question is less than \(n\);  and
\item   \(N-r+m+a-1 \geq n\), i.e., when the cohomological degree
is larger than or equal to \(n\).
\end{enumerate}

Case (1) is handled by induction.
In this case, each \(\bigcap_{i\in I} X_i\) is cut out in
\(\mathbb{A}^{N}\) by \((r-a)\) total defining equations, and there
are \((r-a)-(N-n)\) ``extra equations''.  Write
\(N-r+m+a-1 = N - (r-a) +m-1\).  The inductive hypothesis shows that
the colevel of
\(\mathrm{H}^{N-r+m+a-1}_{c}(\bigcap_{i\in I}X_i)\) is at
least \(\geq \nu_{0}^{((r-a)-(m-1))}(N;d_1,\ldots,d_{r-a})\).  By
Lemma~\ref{lemma:properties-of-nu}(\ref{property-2}), this number is
\(\geq \nu_0^{(r-m)}(N;d_1,\ldots,d_{r-a})\).  This latter number is
trivially \(\geq \nu_0^{(r-m)}(N;d_1,\ldots,d_{r})\).

Case (2) is handled by Theorem \ref{theorem:beyond}.  Since
\(N-r+m+a-1 \geq n\), we write it as \(n + (N-n-r+m+a-1)\).  By
Theorem~\ref{theorem:beyond}, a lower bound of the colevel of
\(\mathrm{H}^{N-r+m+a-1}_{c}(\bigcap_{i\in I} X_i)\) is
\(\nu_{N-n-r+m+a-1}^{(N-n)}(N;d_1,\ldots,d_{r-a})\).  By
Lemma~\ref{lemma:properties-of-nu}(\ref{property-3}), this number is
\(\geq \nu_{0}^{(r-m-a+1)}(N;d_1,\ldots,d_r)\).  Since \(0 \geq 1-a\),
and since \(r \geq r - (N-n) \geq m\) by assumption,
\(r-m-a+1 \geq 1\).  Thus,
Lemma~\ref{lemma:properties-of-nu}(\ref{property-2}) implies that the
above number is \(\geq \nu_0^{(r-m)}(N;d_1,\ldots,d_r)\).

\medskip
\textbf{\textit{Situation II.}} \(a=r-(N-n)\).

This is when the intersection is empty, so we need to show the colevel
of \(\mathrm{H}^{n+m-1}_{c}(\bigcup X_i)\) is at least
\(\nu_0^{(r-m)}(N;d_1,\ldots,d_r)\).  When \(m = 0\), this is the
Ax--Katz bound, so we are reduced to
Theorem~\ref{theorem:previous}(\ref{item-1a}).  If \(m \geq 1\), by
Lemma~\ref{lemma:almost-complete-intersection-colevel}, the colevel of
\(\mathrm{H}^{n+m-1}_{c}(\bigcup X_i)\) is at least
\(\nu_{m-1}^{(r-1)}(N;d_1,\ldots,d_r)\).  By
Lemma~\ref{lemma:properties-of-nu}(\ref{property-3}), this number is
\(\geq\nu_{0}^{((r-1)-(m-1))}(N;d_1,\ldots,d_r)=\nu_0^{(r-m)}(N;d_1,\ldots,d_r)\),
as desired.  This completes the proof.
\end{proof}

\begin{proof}[Proof of Theorem~\ref{theorem:projective-case}]
It suffices to prove the claim about cohomology of \(Y\); the claim
for \(\mathbb{P}^{N}\setminus Y\) follows from the usual long exact
sequence for relative cohomology. We shall prove the theorem by
induction on \(n = \dim Y\).  The base case (\(n=0\)) is trivial. If
\(n > 0\) and \(n = N\), then the theorem is also trivial. So we may
assume \(0 < n < N\).

Let \(\widehat{Y} \subset \mathbb{A}^{N+1}\) be the affine cone of
\(Y\). Let \(L \to Y\) be the total space of the ``tautological line
bundle'' of \(\mathbb{P}^{N}\) restricted to \(Y\). Thus, \(Y\) embeds
into \(L\) as the zero section, whose complement can be identified with
\(\widehat{Y}\setminus O\).   We have a long exact sequence
\begin{equation}\label{eq:wang}
\cdots \to \mathrm{H}^{i}_{c}(L)
\xrightarrow{u} \mathrm{H}^{i}(Y) \to
\mathrm{H}^{i+1}_{c}(\widehat{Y}\setminus O) \to
\mathrm{H}^{i+1}_{c}(L) \to \cdots
\end{equation}
Since the fibers of \(L\to Y\) are \(\mathbb{A}^1\), its Leray
spectral sequence degenerates into an isomorphism
\(\mathrm{H}^{i}_{c}(L) \simeq
\mathrm{H}^{i-2}(Y;\Lambda(-1))\).  Under this isomorphism, the
mapping \(u\) is identified with the cup product with \(c_1\) mapping
\begin{equation*}
\smile c_1(\mathcal{O}_Y(-1))\colon
\mathrm{H}^{i-2}(Y;\Lambda(-1)) \to \mathrm{H}^i(Y;\Lambda).
\end{equation*}

Moreover, for any \(i \geq 1\), the natural morphism
\[
\mathrm{H}^{i+1}_{c}(\widehat{Y}\setminus O)
\xrightarrow{\sim}  \mathrm{H}^{i+1}_{c}(\widehat{Y})
\]
is an isomorphism.  Therefore, for \(i \geq 1\), the exact sequence
\eqref{eq:wang} can be rewritten as
\begin{equation}\label{eq:wang-bis}
\cdots \to \mathrm{H}^{i-2}(Y;\Lambda(-1))_{\mathrm{prim}}
\xrightarrow{\smile c_1(\mathcal{O}_Y(1))}
\mathrm{H}^i(Y)_{\mathrm{prim}} \to \mathrm{H}^{i+1}_{c}(\widehat{Y}) \to \cdots.
\end{equation}
Here, we have made two small modifications.  First, we have decided to take
the cup product with \(c_{1}(\mathcal{O}_{Y}(1))=-c_{1}(\mathcal{O}_{Y}(-1))\).
Obviously this does not affect the exactness.  Second, we have taken the
primitive part.  This modification preserves exactness because the non-primitive
classes are precisely the powers of \(c_{1}(\mathcal{O}_{Y}(1))\).

By Theorems~\ref{theorem:beyond} and \ref{theorem:before},
\begin{itemize}
\item if \(i = N-r+m\) for some \(0\leq m \leq r - (N-n)\), then the
colevel of \(\mathrm{H}^{i+1}_{c}(\widehat{Y})\) is at least
\(\nu^{(r-m)}_0(N+1;d_1,\ldots,d_r)\);
\item if \(i = n + j\), for some \(j \geq 1\), then the colevel of
\(\mathrm{H}^{i+1}_{c}(\widehat{Y})\) is at least
\(\nu^{(N-n)}_j(N+1;d_1,\ldots,d_r)\).
\end{itemize}
Thus, by the exactness of \eqref{eq:wang-bis}, it suffices to prove
that the colevel of the image of the mapping
\begin{equation}\label{eq:chern}
\mathrm{H}^{i-2}(Y;\Lambda(-1))_{\mathrm{prim}} \xrightarrow{\smile c_1(\mathcal{O}_Y(1))} \mathrm{H}^i(Y)_{\mathrm{prim}}
\end{equation}
is at least
\begin{itemize}[left=2.5em]
\item [(a)] \(\nu_0^{(r-m)}(N+1;d_1,\ldots,d_r)\), if \(i = N-r + m\)
for some \(1\leq m \leq r- (N-n)\) (the \(m=0\) case already being taken
care of by Theorem~\ref{theorem:previous}(\ref{item-2a})), and
\item [(b)] \(\nu_j^{(N-n)}(N+1;d_1,\ldots,d_r)\), if \(i = n + j\) for some
\(j \geq 1\).
\end{itemize}

Let \(A\) be any hyperplane in \(\mathbb{P}^N\).  We claim that the
cup product mapping be factored as
\begin{equation}
\label{eq:factorization-chern}
\begin{tikzcd}
\mathrm{H}^{i-2}(Y;\Lambda(-1)) \ar[r]
\ar[rr,bend right=25,"\smile c_1(\mathcal{O}_Y(1))"] &
\mathrm{H}^{i-2}(Y \cap A ; \Lambda(-1)) \ar[r] & \mathrm{H}^i(Y)
\end{tikzcd}
\end{equation}
where the left arrow is the restriction map.  For the moment, let us
grant the existence of the factorization
\eqref{eq:factorization-chern}, and see how it implies the theorem.
Since \(A\) can be arbitrarily chosen, we shall choose it so that
\(\dim Y \cap A = \dim Y - 1 = n - 1\).  Viewing \(A\) as an \(N-1\)
dimensional projective space, we can apply the inductive hypotheses to
the \((n-1)\)-dimensional scheme \(Y \cap A\).  Therefore, the colevel
of \(\mathrm{H}^{i-2}(Y \cap A ; \Lambda(-1))_{\mathrm{prim}}\) is at
least
\begin{itemize}[left=2.5em]
\item [($\alpha$)] \(1 + \nu_0^{(r-(m-1))}(N+1;d_1,\ldots,d_r)\) if
\(i = N-r+m\) (that is, \(i-2 = (N-1) -r + (m-1)\)), for
\(1\leq m \leq r- (N-n)\);
\item [(\(\beta\))]\(1 + \nu_{j-1}^{((N-1)-(n-1))}(N;d_1,\ldots,d_r)\), if
\(i = n+j\) (that is, \(i-2 = (n-1) + (j-1)\)), \(j\geq 1\).
\end{itemize}
The induction proof is concluded by observing the following elementary
facts:
\begin{itemize}[left=2.5em]
\item [(0)] \eqref{eq:factorization-chern} gives a factorization
\[
\mathrm{H}^{i-2}(Y;\Lambda(-1))_{\mathrm{prim}}
\to \mathrm{H}^{i-2}(Y\cap A;\Lambda(-1))_{\mathrm{prim}} \to
\mathrm{H}^{i}(Y)_{\mathrm{prim}};
\]
of \eqref{eq:chern}
\item [(\(1_{e}\))]
\(1 + \nu_{0}^{(e+1)}(N;d_{1},\ldots,d_{r}) \geq \nu_{0}^{(e)}(N+1;d_{1},\ldots,d_{r})\),
for any \(e \in \mathbb{Z}\), and
\item [(\(2_{e}\))]
\(1 + \nu^{(e)}_{j-1}(N;d_{1},\ldots,d_{r}) =\nu_{j}^{(e)}(N+1;d_{1},\ldots,d_{r})\)
for any \(j \geq 1\).
\end{itemize}

Indeed, given Item (0), (\(\alpha\)) and ($1_{r-m}$) imply (a), and
($\beta$) and (\(2_{N-n}\)) imply (b).  Item (0) is straightforward,
and Item (2) can be seen simply by expanding the definition.  To prove
(1), we note that by (2), we have
\(1 + \nu_{0}^{(e+1)}(N;d_{1},\ldots,d_{r}) =
\nu_{1}^{(e)}(N;d_{1},\ldots,d_{r})\).  Then (1) follows from
Lemma~\ref{lemma:properties-of-nu}(\ref{property-3}).
\end{proof}

It remains to explain the existence of the factorization
\eqref{eq:factorization-chern}.

\begin{proof}[Existence of \eqref{eq:factorization-chern}]
The Chern class
\(c_1(\mathcal{O}_{\mathbb{P}^N}(1)) \in \mathrm{H}^2(\mathbb{P}^N;\Lambda(1))\)
can be regarded as a mapping, in the derived category of constructible
complexes of \(\Lambda\)-sheaves
\begin{equation}\label{eq:sheaf-level-cup-chern-pn}
\Lambda_{\mathbb{P}^N}(-1)[-2] \xrightarrow{\smile c_1(\mathcal{O}_{\mathbb{P}^N}(1))} \Lambda_{\mathbb{P}^N}.
\end{equation}

Let \(A\) be a hyperplane in \(\mathbb{P}^N\), and let \(\iota\colon A \to \mathbb{P}^N\) be
the inclusion morphism.  The relative purity theorem implies that
\(\iota^! \Lambda_{\mathbb{P}^N} = \Lambda_{A}(-1)[-2]\).  From this, we obtain a
``Gysin morphism'' \(\iota_\ast \Lambda_A(-1)[-2] \xrightarrow{\mathrm{Gys}}
\Lambda_{\mathbb{P}^N}\) arising from the counit morphism \(\iota_\ast\iota^!
\Lambda_{\mathbb{P}^N} \to \Lambda_{\mathbb{P}^N}\).  The mapping
\eqref{eq:sheaf-level-cup-chern-pn} can be identified with
\begin{equation}
\label{eq:sheaf-level-factorization-pn}
\begin{tikzcd}
\Lambda_{\mathbb{P}^N}(-1)[-2] \ar[r] \ar[rr,bend right=25,"\smile c_1(\mathcal{O}_{\mathbb{P}^N}(1))"] &
\iota_\ast  \Lambda_A (-1) [-2]
\ar[r,"\mathrm{Gys}"] &
\Lambda_{\mathbb{P}^N}.
\end{tikzcd}
\end{equation}
This identification holds because the cycle class of \(A\) coincides with
\(c_{1}(\mathcal{O}_{\mathbb{P}^{N}}(1))\).

Since taking the Chern class is compatible with the restriction of
invertible sheaves, applying the \(\ast\)-inverse image to
\eqref{eq:sheaf-level-cup-chern-pn} with respect to
\(Y \hookrightarrow \mathbb{P}^{N}\) gives rise to the Chern class of
\(\mathcal{O}_{Y}(1)\).  By restricting
\eqref{eq:sheaf-level-factorization-pn} to \(Y\) and applying the
proper base change theorem, we obtain the following factorization:
\begin{equation}
\label{eq:sheaf-level-factorization-y}
\begin{tikzcd}
\Lambda_{Y}(-1)[-2] \ar[r] \ar[rr,bend right=25,"\smile c_1(\mathcal{O}_{Y}(1))"] &
\iota^{\prime}_\ast  \Lambda_{Y\cap A} (-1) [-2]
\ar[r,"\mathrm{Gys}"] &
\Lambda_{Y},
\end{tikzcd}
\end{equation}
where \(\iota^{\prime}\colon Y\cap A \to Y\) denotes the inclusion morphism.

Applying \(R\Gamma(Y;-)\) to the sheaf-theoretic factorization
\eqref{eq:sheaf-level-factorization-y} gives the cohomological
factorization \eqref{eq:factorization-chern}, as claimed.
\end{proof}

\begin{remark}
A version of this proof appeared in \cite{wz22}. However, the argument
there contained a gap: it incorrectly claimed that the image of $u$ in
\eqref{eq:wang} was always ``non-primitive''.  The proof presented
above corrects this oversight: although the image of $u$ may indeed be
primitive, it arises as the Gysin image of a hyperplane section, whose
colevel can be controlled through induction. We thank an anonymous
referee for bringing this issue to our attention.
\end{remark}

\section{Two examples}

We end with two ``determinantal'' examples illustrating how the new bounds
provided by Theorems~\ref{theorem:beyond} and \ref{theorem:before} are sharper
than the previously known (or anticipated) bounds.

\begin{example}%
\label{example:det-1}
Let \(e\) be a positive integer \(\geq 4\).
Let \(N = e^{2}\). View \(\mathbb{A}^N\) as the space of \(e\times e\)-matrices.
Consider the following variety
\begin{equation*}
X = \{A \in \mathbb{A}^{N}:  \text{ all principal } (e-1) \text{-minors are zero}\}.
\end{equation*}
Then \(X\) is defined by the vanishing of \(r = e\) equations of degree \(e-1\).
The Ax--Katz bound implies that the colevels of \(\mathrm{H}^{\ast}_{c}(X)\) are \(\geq 2\).

According to Wheeler~\cite[Theorem 4]{wheeler},  \(X\) has two irreducible
components: \(X = X_{1} \cup X_{2}\), where \(X_1 = \{A \in \mathbb{A}^{N}:\operatorname{rank}A\leq e-2\}\),
\(X_{2}\) is the Zariski closure of \(\{A \in X: \det A \neq 0\}\).
Moreover, \(\dim X_2 = N-e\). On the other hand, it is
well-known that \(\dim X_{1} = e^2 - 4\) (\cite[p.~67, Proposition]{acgh}). Thus the dimension of \(X\) is \(n=e^2-4\).
The affine variety $X$ in $\mathbb{A}^N$ is a complete intersection if $e=4$, not a complete intersection if $e\geq 5$.

The improved bounds, Theorem~\ref{theorem:beyond} and
\ref{theorem:before}, give the following:
\begin{itemize}
\item (before middle dimension) for \(0\leq m \leq e-4\), the colevel
of \(\mathrm{H}^{e^2-e+m}_{c}(X)\) is at least
\(\left\lceil m+1 - \frac{m-1}{e-1} \right\rceil\) which is $m+1$ if
$1\leq m\leq e-4$, and $2$ if $m=0$ (Theorem~\ref{theorem:previous}
only shows that these colevels are \(\geq 2\));
\item (beyond middle dimension) for \(0\leq j \leq e^{2}-4\), the
colevel of \(\mathrm{H}^{e^2-4+j}_{c}(X)\) is at least
\(j+\max\left\{0,e-4+\left\lceil \frac{4-j}{e-1}
\right\rceil\right\}\).
When \(e \geq 5\), this latter value is $j+ e-3$ if
$0\leq j\leq 3$, $j+e-4$ if $4\leq j \leq e+2$, $j+e-5$ if
$e+3\leq j \leq 2e+1$, etc.  By contrast, the colevel lower bound
anticipated by Esnault and the first author
(Theorem~\ref{theorem:positive}) is
\(j + \max\left\{ 0, \left\lceil
\frac{e-j}{e-1}\right\rceil\right\}\), which is $j+2$ if $j =0$,
\(j+1\) if $1\leq j \leq e-1$ and \(j\) if $j\geq e$.
\end{itemize}
\end{example}

\begin{example}%
\label{example:det-2}
Notation as Example~\ref{example:det-1}, consider
\begin{equation*}
X^{\prime} = \{A \in \mathbb{A}^N: \operatorname{rank} A \leq e-1, \text{ all principal } (e-1) \text{-minors are zero}\}.
\end{equation*}
Then \(X^{\prime}\) is cut out by \(r=e+1\) equations: the vanishing
of the \(e\) principal minors (degree \(e-1\) polynomials),
and \(\det A = 0\) (a degree \(e\) polynomial).
Thus \(d_1 = e\), \(d_2 = \ldots = d_{r} = e-1\).

From Wheeler's theorem cited above it is easy to see that
we can write \(X^{\prime} = X_{1} \cup X^{\prime}_{2}\), where \(X^{\prime}_{2} = X_{2} \cap \{\det = 0\}\).
Thus \(X^{\prime}_{2}\) has codimension \(e+1\geq 5\) in \(\mathbb{A}^{N}\).
It follows that $X^{\prime}$ has codimension $4$ in \(\mathbb{A}^{N}\).

Since \(\frac{N-\sum d_{i}}{d_{1}} = 0\),
the bounds given by Theorem~\ref{theorem:positive}  do not
improve (and is the same as) Deligne's bounds, Theorem~\ref{theorem:deligne}.
By contrast, Theorems~\ref{theorem:beyond} and \ref{theorem:before} give the following better lower bounds.
\begin{itemize}
\item Before middle dimension, for \(0 \leq m \leq e - 3\),  we have:
\begin{equation*}
\text{colevel of } \mathrm{H}^{e^2-(e+1)+m}_{c}(X^{\prime}) \geq
\left\lceil e-1  - \frac{e-1}{e}(e-m)\right\rceil =  \left\lceil \frac{(e-1)m}{e}\right\rceil.
\end{equation*}
The Esnault--Katz colevel lower bound in this case is $0$.
\item Beyond middle dimension, for $0\leq j \leq e^2-4$, we have
\begin{align*}
  \text{colevel of }\mathrm{H}_{c}^{e^{2}-4+j}(X^{\prime})
  &\geq  j + \max \left\{ 0 , \left\lceil \frac{e^{2}-4e+3-j}{e}\right\rceil\right\} \\
  &= j + \max\left\{ 0, e-4+\left\lceil \frac{3-j}{e}\right\rceil \right\}.
\end{align*}
When \(e \geq 5\), this latter value is $j+e-3$ if $0\leq j \leq 2$, $j+e-4$ if $3\leq j \leq e+2$, $j+e-5$ if $e+3\leq j \leq 2e+2$ etc.
By contrast, the colevel lower bound given by Theorem~\ref{theorem:positive} is $j$.
\end{itemize}
\end{example}

\section*{Acknowledgment}
As noted, Theorem~\ref{theorem:positive} has also been independently
obtained by Rai and Shuddhodan as a by-product of~\cite{rs23}, using a
similar Lefschetz approach.

We are very grateful to an anonymous referee for identifying an error in the proof
of Theorem~\ref{theorem:projective-case} in an earlier draft of the
paper.

DZ thanks Hélène Esnault for her suggestion---made during the AMS
Summer Institute in Salt Lake City---to explore higher divisibility of
Frobenius eigenvalues in the degenerate case.  Both authors express
their gratitude to Esnault for her useful correspondence.

\small
\bibliographystyle{amsalpha}
\bibliography{ew}%

\footnotesize

\textsc{Daqing Wan}\\
\textsc{Center for Discrete Mathematics, Chongqing University, Chongqing 401331 China.}\\
\textsc{Department of Mathematics, University of California, Irvine, CA
92697-3875, USA.}\\
\textit{Email Address:} \texttt{dwan@math.uci.edu}

\textsc{Dingxin Zhang} \\
\textsc{YMSC, Tsinghua University, Beijing 100086, China.}\\
\textit{Email Address:} \texttt{dingxin@tsinghua.edu.cn}

\end{document}